\documentclass[reqno,english]{amsart}
\usepackage{amsfonts,amsmath,latexsym,verbatim,amscd,mathrsfs,color,array}

\usepackage{amsmath,amssymb,amsthm,amsfonts,graphicx,color}
\usepackage{amssymb}
\usepackage{pdfsync}
\usepackage{epstopdf}
\usepackage[colorlinks=true]{hyperref}
\usepackage{subfigure}

\makeatletter
\def\@currentlabel{2.1}\label{e:dispaa}
\def\@currentlabel{2.21}\label{e:dispau}
\def\@currentlabel{2.22}\label{e:dispav}
\def\@currentlabel{2.23}\label{e:dispaw}
\def\@currentlabel{2.24}\label{e:dispax}
\def\theequation{\thesection.\@arabic\c@equation}
\makeatother
\let\oldbibliography\thebibliography
\renewcommand{\thebibliography}[1]{%
\oldbibliography{#1}%
\setlength{\itemsep}{0pt}%
}

\renewcommand{\theequation}{\thesection.\arabic{equation}}
\newtheorem{lemma}{Lemma}[section]

\newtheorem{definition}{Definition}

\newtheorem{proposition}{Proposition}[section]
\newtheorem{corollary}{Corollary}[section]
\newtheorem{remark}{Remark}[section]
\newcommand{\bremark}{\begin{remark} \em}
\newcommand{\eremark}{\end{remark} }

\newcommand{\im}{\operatorname{Im}}

\newtheorem{theorem}{Theorem}[section]

\newcommand{\R}{{\mathbb R}}

\newcommand{\BE}{\begin{equation}}
\newcommand{\BEN}{\begin{equation*}}
\newcommand{\EE}{\end{equation}}
\newcommand{\EEN}{\end{equation*}}
\newcommand{\BL}{\begin{lemma}}
\newcommand{\EL}{\end{lemma}}
\newcommand{\BT}{\begin{theorem}}
\newcommand{\ET}{\end{theorem}}
\newcommand{\BP}{\begin{proposition}}
\newcommand{\EP}{\end{proposition}}
\newcommand{\BC}{\begin{corollary}}
\newcommand{\EC}{\end{corollary}}
\renewcommand{\Re}{\operatorname{Re}}
\renewcommand{\Im}{\operatorname{Im}}

\begin{document}

%\title[Green Function]{Two-component reduced functional}
%%%%%%%%%%%%%%%%%%%%%%%%%%%%%%%%%%%%%%%%%%%%%%%%%%%%%%%%%%%%%%%%%%%%%%

\title[Minimum principle]{\bf On universally optimal lattice phase transitions and energy minimizers of completely monotone potentials}

\author{Senping Luo}

\author{Juncheng Wei}

\author{Wenming Zou}

\address[S.~Luo]{School of Mathematics and statistics, Jiangxi Normal University, Nanchang, 330022, China}
\address[J.~Wei]{Department of Mathematics, University of British Columbia, Vancouver, B.C., Canada, V6T 1Z2}
\address[W.~Zou]{Department of Mathematical Sciences, Tsinghua University, Beijing 100084, China}

\email[S.~Luo]{luosp1989@163.com or spluo@jxnu.edu.cn}

\email[J.~Wei]{jcwei@math.ubc.ca}

\email[W.~Zou]{zou-wm@mail.tsinghua.edu.cn}

\begin{abstract}
We consider the minimizing problem for energy functionals with two types of competing particles and completely monotone potential on a lattice.  We prove that the minima of sum of two completely monotone functions among lattices is located exactly on a special curve which is part of the boundary of the fundamental region. We also establish  a  universal result for square lattice being the optimal in certain interval, which is surprising. Our result establishes  the hexagonal-rhombic-square-rectangular transition lattice shapes in many physical and biological system (such as Bose-Einstein condensates and two-component Ginzburg-Landau systems). It turns out, our results also apply to locating the minimizers of sum of two Eisenstein series, which is new in number theory.

\end{abstract}

\maketitle

\setcounter{equation}{0}

\section{Introduction and main results}

\subsection{Lattice energy functional}
We first introduce the energy functional considered in this paper.
Consider a set of $N$ classical identical particles in $\R^d$ ($d\leq 3$), interacting by pairs through a potential $\mathcal{V}$ depending on the distance between them. Let $x_1, x_2,\cdots, x_N\in \R^d$ and $p_1,p_2,\cdots,p_N\in\R^d$ be the positions and momenta of these particles.

The Hamiltonian dynamics of the model, based on the energy, is given by

\begin{equation}
\mathcal{H}_N(x_1,\cdots, x_N,p_1,\cdots, p_N)=\sum_{i=1}^N\frac{|p_i|^2}{2m}+\sum_{1\leq i<j\leq N} \mathcal{V}(|x_i-x_j|^2).
\end{equation}
Here $m$ is the mass of the particles and $|\cdot|$ is the Euclidean norm of $\R^d$.
When the temperature is zero, the equilibrium states are the minima of the functional $\mathcal{H}_N$, which satisfy $p_1=\cdots=p_N$. If one is interested in these equilibrium states, then it suffices to consider the potential energy
 \begin{equation}\label{VN}
\mathcal{E}(\{x_i\}_{i=1}^N):=\sum_{1\leq i\neq j\leq N}\mathcal{V}(|x_i-x_j|^2),
\end{equation}
where $\mathcal{V}$ is referred as the potential between two particles.

\medskip

Let $\Lambda$ be a two-dimensional lattice generated by the two basis $\alpha_1$ and $\alpha_2$. Assume that there are two types of competing  particles $A$ and $B$.
Now we consider the dynamics of two dimensional periodical lattices alternated by type A particle and type B particle (see figure \ref{f-assemblies}), where one is
located at the center of another. That is, we only consider a special type of $(\Omega_1, \Omega_2)$,
termed two species periodic assemblies of discs, denoted by
$(\Omega_{\alpha,1}, \Omega_{\alpha,2})$, with
\begin{align}
  \Omega_{\alpha, 1}& = \bigcup_{\lambda \in \Lambda}
   \Big \{ B(\xi, r_1) \cup B(\xi', r_1):
  \xi = \frac{3}{4} \alpha_1
     + \frac{1}{4}\alpha_2 + \lambda, \
     \xi' = \frac{1}{4} \alpha_1
     + \frac{3}{4} \alpha_2 + \lambda \Big \},
  \label{Omegaalpha1} \\
  \Omega_{\alpha, 2}& = \bigcup_{\lambda \in \Lambda}
  \Big \{ B(\xi, r_2) \cup B(\xi', r_2):
  \xi = \frac{1}{4} \alpha_1
     + \frac{1}{4} \alpha_2 + \lambda , \
     \xi' = \frac{3}{4}  \alpha_1
     + \frac{3}{4} \alpha_2 + \lambda
     \Big \}.   \label{Omegaalpha2}
\end{align}

   The A type particles are located on red spots, denoted by $P_{4k+1}, P_{4k+3}$; the B type particles are located on yellow spots, denoted by $P_{4k+2}, P_{4k+4}$, where $k\in N$. Let $\mathcal{V}$ be the potential of the system. Then the energy  \eqref{VN} of the system consists of two parts, the self-interaction part of each type particles and the interaction between the two different types of particles, while these two parts contribute to the total energy in distribution of  $\kappa$ and $1-\kappa$, where $\kappa\in(0,1)$.

\begin{figure}
\centering
 \includegraphics[scale=0.08]{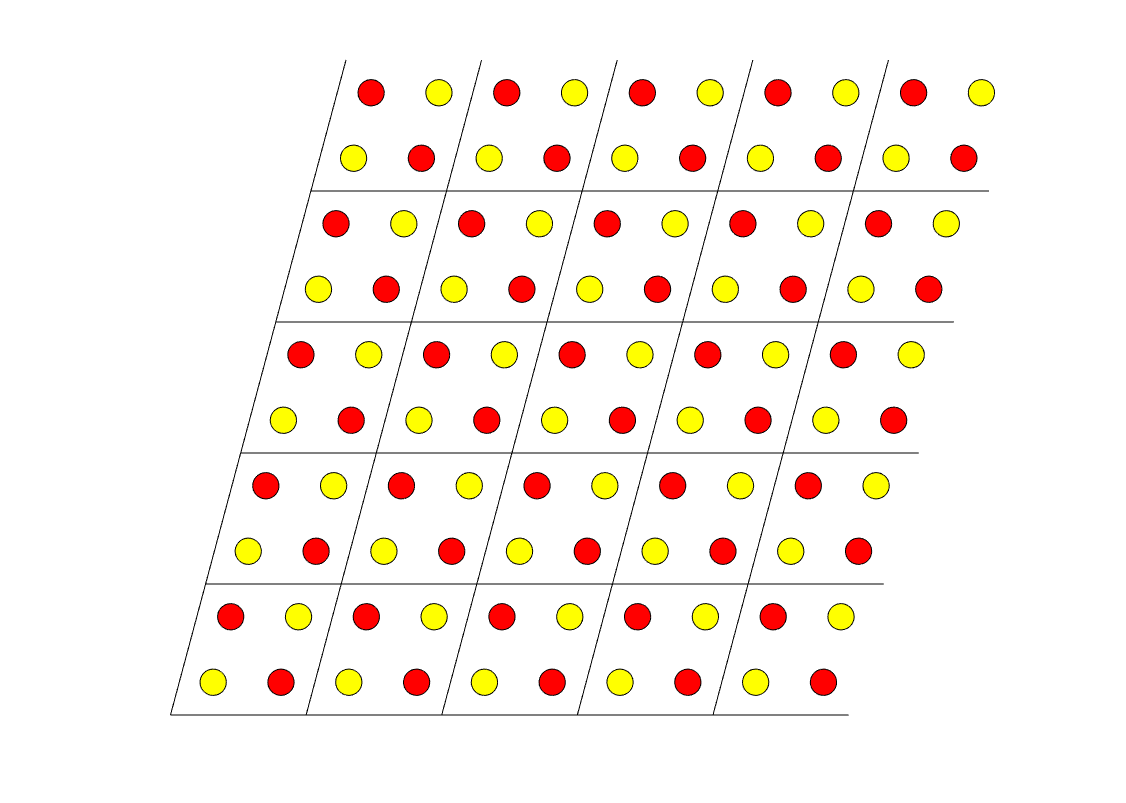}\includegraphics[scale=0.08]{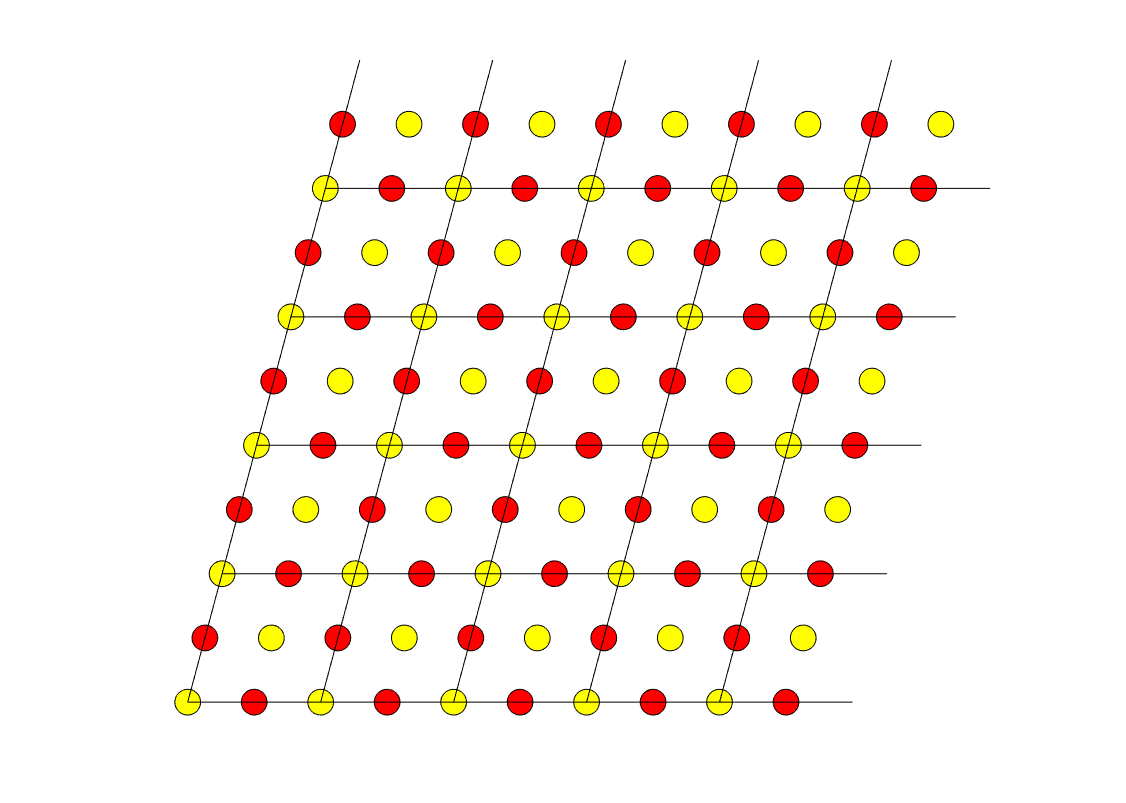}
 \caption{Two lattices with  centers at the lattice points and the half lattice
   points.}
\label{f-assemblies}
\end{figure}

Since the system is an infinite particles system, we first consider a mix of $N-$particles $(\{x_i\}_{i=1}^N)$ of type A and $N-$particles $(\{y_i\}_{i=1}^N)$ of type B to approximate the infinite particles system. For convenience, we set $N=M^2$. The energy of self interaction of the same type particles is
 \begin{equation}\nonumber\aligned
\mathcal{E}_{a}:=\sum_{(k,j)\in S_N\times S_N\setminus\{(0,0)\}}\frac{1}{2}\mathcal{V}(|k\alpha_1+j\alpha_2|^2){(M-|k|)(M-|j|)},
\endaligned\end{equation}
where the set $S_N$ is defined by $\{-(M-1),-(M-2),\cdots,0,\cdots,M-2, M-1\}$.
The limit per-energy of the self interaction of the same type particles is
  \begin{equation}\nonumber\aligned
&\mathcal{E}_{per,1}({\{x_j\}_{j=1}^\infty})=\mathcal{E}_{per,1}({\{y_j\}_{j=1}^\infty}):=\lim_{N\rightarrow\infty}\frac{\mathcal{E}_{a}}{N}\\
=&\lim_{M\rightarrow\infty}\sum_{(k,j)\in S_N\times S_N\setminus\{(0,0)\}}\mathcal{V}(|k\alpha_1+j\alpha_2|^2)\frac{(M-|k|)(M-|j|)}{2M^2}.
\endaligned\end{equation}
Such a limit converges if the summation $\sum_{(k,j)\in S_N\times S_N\setminus\{(0,0)\}}\mathcal{V}(|k\alpha_1+j\alpha_2|^2)$ is summable. One has
  \begin{equation}\nonumber\aligned
\mathcal{E}_{per,1}({\{x_j\}_{j=1}^\infty})=\mathcal{E}_{per,1}({\{y_j\}_{j=1}^\infty})=\frac{1}{2}\sum_{(k,j)\neq\{(0,0)\}}\mathcal{V}(|k\alpha_1+j\alpha_2|^2).
\endaligned\end{equation}
One can rewrite it in a concise form on the lattice $\Lambda$
 \begin{equation}\nonumber
\mathcal{E}_{per,1}({\{x_j\}_{j=1}^\infty})=\mathcal{E}_{per,1}({\{y_j\}_{j=1}^\infty})=\frac{1}{2}\sum_{\mathbb{P}\in\Lambda\backslash \{0\}}  \mathcal{V}(|\mathbb{P}|^2).
\end{equation}

      Next we consider the interaction between two different particles. The energy between $N-$particles$(\{x_i\}_{i=1}^N)$ of type A and $N-$particles$(\{y_i\}_{i=1}^N)$ of type B is
 \begin{equation}\nonumber\aligned
\mathcal{E}_{b}:=\sum_{(k,j)\in S_N\times S_N\setminus\{(0,0)\}}\mathcal{V}(|\frac{1}{2}(\alpha_1-\alpha_2)+k\alpha_1+j\alpha_2|^2){(M-|k|)(M-|j|)}.
\endaligned\end{equation}

The limit per-energy of the interaction between two different type particles $A,B$ is determined by
 \begin{equation}\nonumber\aligned
\mathcal{E}_{per,2}({\{x_j\}_{j=1}^\infty},{\{y_j\}_{j=1}^\infty} ):&=
\lim_{N\rightarrow\infty}\frac{\mathcal{E}_{b}}{N}\\
&=\sum_{(k,j)\in \mathbb{Z}\bigotimes \mathbb{Z}}\mathcal{V}(|\frac{1}{2}(\alpha_1-\alpha_2)+k\alpha_1+j\alpha_2|^2).
\endaligned\end{equation}
One can rewrite it as
 \begin{equation}\nonumber\aligned
\mathcal{E}_{per,2}({\{x_j\}_{j=1}^\infty},{\{y_j\}_{j=1}^\infty} )=\sum_{\mathbb{P}\in\Lambda}  \mathcal{V}(|\frac{\mathbb{P}+1}{2}|^2).
\endaligned\end{equation}
Therefore, the total energy of two dimensional periodical alternating system with contribution $\kappa, 1-\kappa$ of same-type and different-types interaction is determined by
 \begin{equation}\label{LatticeEnergy}\aligned
\mathcal{E}_{per}(A,B):&=\kappa(\mathcal{E}_{per,1}({\{x_j\}_{j=1}^\infty})+\mathcal{E}_{per,1}({\{y_j\}_{j=1}^\infty}))+(1-\kappa)\mathcal{E}_{per,2}({\{x_j\}_{j=1}^\infty},{\{y_j\}_{j=1}^\infty} )\\
&=\kappa\sum_{\mathbb{P}\in\Lambda\backslash \{0\}}  \mathcal{V}(|\mathbb{P}|^2)
+(1-\kappa)\sum_{\mathbb{P}\in\Lambda}  \mathcal{V}(|\frac{\mathbb{P}+1}{2}|^2),
\endaligned\end{equation}
where $\kappa\in [0,1]$ and we include the endpoint cases $\kappa=0,1$ for convenience.
To find out which shape of lattice $\Lambda$ minimizes the lattice energy \eqref{LatticeEnergy}, we consider the minimum of the  functional $ \mathcal{E}_{per} (A, B)$ over $\Lambda$, which is the main contribution of this paper.

\subsection{Sum of two completely monotone functions on the lattice and their minimums}

We now consider the potential in the lattice energy functional with special form, namely, the completely monotone potential cases.
A function $f:(0,\infty)\rightarrow\R$ is completely monotone if it is of class $C^\infty(0,\infty)$ and
$$
(-1)^j f^{(j)}(x)>0, j=0,1,2,\cdots \infty.
$$

The Bernstein functions are non-negative functions whose derivative is completely
monotone. They are such that
$$
f(x)>0, (-1)^{j-1} f^{(j)}(x)>0, j=1,2,\cdots \infty.
$$

The Bernstein functions have rich connection to other fields, firstly in probability theory.
Complete Bernstein functions are used in complex analysis under the name Pick or Nevanlinna
functions, while in matrix analysis and operator theory, the name operator monotone
function is more common. When studying the positivity of solutions of Volterra integral
equations, various types of kernels appear which are related to Bernstein functions. See the monograph \cite{Song2012}.

%%%%%%%%%%%%%%%%%%%%%%%%%%%%%%%%%%%%%%%%%
Let $\mathcal{F}(x)$ be completely monotone function and $\Lambda$ be a two dimensional lattice with fixed area. In (\ref{LatticeEnergy}) we choose the potential $ \mathcal{V}=\mathcal{F}$. Now we consider the lattice summation
 \begin{equation}\label{F0}
\mathcal{E}(\kappa,\Lambda)=\sum_{\mathbb{P}\in\Lambda\backslash \{0\}} \Big((1-\kappa)\mathcal{F}(|\frac{\mathbb{P}+1}{2}|^2)+\kappa \mathcal{F}(|\mathbb{P}|^2)\Big).
\end{equation}
The main aim of this paper is to study the optimal lattice shape which minimizes  $\mathcal{E}(\kappa,\Lambda)$, where $\kappa\in[0,1]$ is a parameter.

To write $\mathcal{E}(\kappa,\Lambda)$ in function form, we parametrize the lattice $\Lambda$ with cell  by $ \Lambda =\sqrt{\frac{\pi}{\Im(z)}}\Big({\mathbb Z}\oplus z{\mathbb Z}\Big)$ where the only parameter is $z$ which belongs to $ \mathbb{H}:=\{z= x+ i y\in\mathbb{C}: y>0\}$. It follows that
 \begin{equation}\nonumber\aligned
\mathcal{F}(|\mathbb{P}|^2):&=\sum_{(m,n)\in\mathbb{Z}^2\backslash \{0\}}\mathcal{F}(\frac{\pi }{\im(z) }|mz+n|^2),\\
\mathcal{F}(|\frac{\mathbb{P}+1}{2}|^2):&=\sum_{(m,n)\in\mathbb{Z}^2\backslash \{0\}}\mathcal{F}(\frac{\pi }{\im(w) }|mw+n|^2)\mid_{w=\frac{z+1}{2}}.
\endaligned\end{equation}

Therefore, we can write
 \begin{equation}\label{F1}
\mathcal{E}_\kappa(z):=\mathcal{E}(\kappa,\Lambda)=
\sum_{(m,n)\in\mathbb{Z}^2\backslash \{0\}}(1-\kappa)\mathcal{F}(\frac{\pi }{2\Im(z) }|m(z+1)+2n|^2)+\kappa\mathcal{F}(\frac{\pi }{\Im(z) }|mz+n|^2).
\end{equation}

We introduce the special curve in $\mathbb{H}$. Let the vertical interval be
\begin{equation}\aligned\nonumber
\Omega_{ea}:=\{z: \Re(z)=0, 1\leq \Im(z)\leq\sqrt3\},
\endaligned\end{equation}
while the arc on the circle is denoted by
\begin{equation}\aligned\nonumber
\Omega_{eb}:=\{z: |z|=1,0\leq \Re(z)\leq\frac{1}{2}\}.
\endaligned\end{equation}
The union of the vertical interval and the arc is denoted as
\begin{equation}\aligned\label{Curve1}
\Omega_e:=\Omega_{ea}\cup\Omega_{eb}.
\endaligned\end{equation}

In PDE theory, the maximum/minimum principle states that a function satisfying some partial differential equation in some domain achieves its maximum/minimum on the boundary of the domain.
Now we state a minimum principle for $\mathcal{E}_\kappa (z)$ on $ \mathbb{H}$:
\begin{theorem}[{\bf Minimum principle for sum of two completely monotone functions on the lattice}]\label{Th1} Let $\mathcal{F}(x)$ be any completely monotone function, $\kappa\in[0,1]$ be the parameter and  $\mathcal{E}_\kappa(z)$ be defined in \eqref{F1}. Then
\begin{equation}\nonumber
\min_{z\in \mathbb{H}}\mathcal{E}_\kappa(z)=\min_{z\in\Omega_e}\mathcal{E}_\kappa(z).
\end{equation}
\end{theorem}

\begin{remark}\label{Re1} Theorem \ref{Th1} holds for any two nonnegative combination of two completely  monotone functions on the lattice. In fact, consider the nonnegative combination,
 \begin{equation}\nonumber
\mathcal{E}(a,b,\Lambda):=\sum_{\mathbb{P}\in\Lambda\backslash \{0\}} \Big(a\mathcal{F}(|\frac{\mathbb{P}+1}{2}|^2)+b \mathcal{F}(|\mathbb{P}|^2)\Big).
\end{equation}
After normalization, one can rewrite it as
 \begin{equation}\nonumber\aligned
\mathcal{E}(a,b,\Lambda)&=(a+b)\sum_{\mathbb{P}\in\Lambda\backslash \{0\}} \Big(\frac{a}{a+b}\mathcal{F}(|\frac{\mathbb{P}+1}{2}|^2)+\frac{b}{a+b} \mathcal{F}(|\mathbb{P}|^2)\Big)\\
&=(a+b)\mathcal{E}(\frac{b}{a+b},\Lambda),
\endaligned\end{equation}
then it becomes the standard form in Theorem \ref{Th1}. Similar remarks are applied  to the theorems \ref{Th2}, \ref{Th3} and \ref{Th4} below.

\end{remark}

In Theorem \ref{Th1}, we conclude that the minimizers of sum of two completely monotone functions are located precisely on a special curve. As we will see later, this special curve is a partial boundary of the fundamental region under the group defined in \eqref{GroupG2}.

Many of the physical meaningful potentials are completely monotone functions, including Riesz potentials $r^{-a}(a>0),$  screened Coulomb potential (also called Yukawa potential) $a\frac{e^{-br}}{r}(a,b>0)$, and  Born-Mayer potential $ae^{-br}(a,b>0)$.
There are also  a large number of completely monotone functions. Examples include $(a+\frac{b}{x^\alpha})^\mu,\;\hbox{here}\; a\geq0, b\geq0, \mu\geq0, 0\leq\alpha\leq1$; $e^{\frac{a}{x}} (a>0)$, $\frac{1}{(a-b e^{-x})^\mu}, \;\hbox{here}\;a\geq b>0, \mu>0$; $\frac{1}{x^a(\Gamma(1+\frac{1}{x}))^x}, \;\hbox{here}\;a\geq1$; $\frac{\Gamma(x)\Gamma(x+a+b)}{\Gamma(x+a)\Gamma(x+b)} (a, b\geq 0)$; $\frac{2}{\sqrt{\pi x}}\int_0^{\sqrt x}e^{-t^2}dt$; etc. There are rules to generate more completely monotone functions from known ones. For example, if $f(x)$ is completely monotone, then $f^{(2m)}(x), -f^{(2m+1)}(x), e^{ f(x)}$, $f(1-e^{-x})$, and $f(ax^\alpha+b)$($a\geq0, b\geq0, 0\leq\alpha\leq1$) are also completely monotone for $m=1,2,3,\cdots$. We refer to  \cite{Ismail1986,Miller2001,Alzer2002,Alzer2006,Askey1974} and the references therein for more concrete examples and the rules to generate completely monotone functions. For the references on the physical applications of  the completely monotone functions, we refer to B\'{e}termin \cite{Betermin2016} and the references therein.

To introduce next Theorem, we follow  the terminology  of "universally optimal" from Cohn-Kumar \cite{Cohn2006JAMS} as defined by
\begin{definition}[Cohn-Kumar \cite{Cohn2006JAMS}]\label{Def1} A finite subset $\mathcal{C}\subset S^{n-1}$
is universally optimal if it (weakly)
minimizes potential energy among all configurations of $|\mathcal{C}|$ points on $S^{n-1}$
for each completely monotonic potential function.
\end{definition}

%%%%%%%%%%%%%%%%%%%%%%%%%%%%%%%%%%%%%%%%%%%%%%%%%%
In parallel, we shall state our definition of universally optimal among all lattice shapes.
\begin{definition}\label{Def2} A lattice shape(determined by a complex variable $z_0$)
is {\bf universally optimal} if it (weakly)
minimizes potential energy among all configurations of lattice shapes
for each completely monotonic potential function.
\end{definition}

The weak minimizer(s) here means that the minimizer is not necessarily unique.
Now we state our universally optimal result for the parameter locating on some interval with positive measure:
\begin{theorem}[{\bf Universally optimal and an explicitly universal interval for square lattice}]\label{Th2} For any completely monotone potential function $\mathcal{F}$, and for any $\kappa\in[\frac{1}{3},\frac{2}{3}]$,
\begin{equation}\nonumber
\min_{z\in \mathbb{H}}\mathcal{E}_\kappa(z)
\end{equation}
 is attained at $i$, corresponding to the square lattice by our setting and is unique up the group $\mathcal{G}_2$ defined in \eqref{GroupG2}.
\end{theorem}

\begin{remark}
The explicitly universal interval $[\frac{1}{3},\frac{2}{3}]$ is not sharp and can be improved to be optimal by our method through refining the estimates. See Section \ref{Theorem2} below.

\end{remark}
\begin{remark}
The universality of optimal  square lattice has also been observed in a recent paper by  Betermin-Knupfer-Faulhuber \cite{Betermin2021} for  another competing system  with  Gaussian and Riesz interactions.
\end{remark}

Note that there are three factors  in the minimization process of $\min_{z\in \mathbb{H}}\mathcal{E}_\kappa(z)$: one is the function (there are a large class of completely monotone functions), one is parameter $\kappa$ which represents the competing strength of the two-component system, and the lattice parameter $z$.
Theorem \ref{Th2} asserts that the minimizer of sum of two completely monotone functions is always located on $i$ (corresponding to the square lattice) for the parameter $\kappa$ in $[\frac{1}{3},\frac{2}{3}]$. This seems rather  surprising.

\vskip0.2in Recall that the minimizer of the lattice energy functional corresponds to the optimal lattice shape in the competing system.
Therefore, combining Theorem \ref{Th1} and Theorem \ref{Th2}, we have the following phases transition result which has  appeared in many competing systems(like Bose-Einstein system, see in Mueller-Ho \cite{Mue2002} and Luo-Wei \cite{LW}). Nevertheless this result applies to large class of other systems.

\begin{theorem}[{\bf Hexagonal-rhombic-square-rectangular lattice phases transition}]\label{Th3} For any completely monotone potential function $\mathcal{F}$, there exists $\kappa_1<\kappa_2\in(0,1)$ depending on the potential function $\mathcal{F}$ such that
\begin{itemize}
  \item Case A: $\kappa\in[0,\kappa_1]$, $\min_{z\in \mathbb{H}}\mathcal{E}_\kappa(z)=\min_{z\in\Omega_{ea}}\mathcal{E}_\kappa(z)$; the minimizer corresponds to rectangular lattice in this case;
  \item Case B: $\kappa\in[\kappa_1,\kappa_2]$, $Minimizer_{z\in \mathbb{H}}\mathcal{E}_\kappa(z)=i$; the minimizer corresponds to square lattice in this case;
  \item Case C: $\kappa\in[\kappa_2,1]$, $\min_{z\in \mathbb{H}}\mathcal{E}_\kappa(z)=\min_{z\in\Omega_{eb}}\mathcal{E}_\kappa(z)$; the minimizer corresponds to rhombic lattice in this case;
\end{itemize}
For the parameters $\kappa_1,\kappa_2$, for any potential function which is completely monotone,
there has
\begin{itemize}
  \item $\kappa_1\leq\frac{1}{3}$;
  \item $\kappa_2\geq\frac{2}{3}$;
  \item $\kappa_1+\kappa_2=1$;
\end{itemize}
namely, $[\kappa_1, \kappa_2]\supseteq[\frac{1}{3}, \frac{2}{3}]$ in Case B.
In particular, when $\kappa=1$, $Minimizer_{z\in \mathbb{H}}\mathcal{E}_\kappa(z)=\frac{1}{2}+i\frac{\sqrt3}{2}$, which corresponds to hexagonal lattice; when $\kappa=0$, $Minimizer_{z\in \mathbb{H}}\mathcal{E}_\kappa(z)=\sqrt{3}i$, which corresponds to $\sqrt3:1$ rectangular lattice.
\end{theorem}

In physical and biological models, the functional $\mathcal{E}_\kappa(z)$ is referred to the energy of competing periodical system and the parameter $\kappa$ reflects the competing strength of the system. As the parameter $\kappa$ varies continuously, we see from Theorem \ref{Th3} clearly that, the optimal lattice shape goes through hexagonal-rhombic-square-rectangular lattice continuously (as $\kappa$ moves from $1$ to $0$  continuously). Theorem \ref{Th2} shows that  under the completely monotone potentials  and under certain positive interval of interaction strength, the minimal energy lattice shape is the square lattice shape. This phenomenon has appeared and conjectured in two component Bose-Einstein condensates (see in Mueller-Ho \cite{Mue2002} and Luo-Wei \cite{LW}). When $\kappa=1$ the optimality of hexagonal lattice for sum of Gaussian lattice has been investigated recently in \cite{Betermin2021-2}.

\medskip

The key tool in proving Theorems \ref{Th1}-\ref{Th3} is the so-called  Bernstein representation formula for general completely monotone function $ \mathcal{F}$. To introduce the formula we define
 \begin{equation}
 \label{Wzalphadef}
\mathcal{W}_\kappa(z;\alpha)=
\sum_{(m,n)\in\mathbb{Z}^2\backslash \{0\}}(1-\kappa)\exp(-\alpha\frac{\pi }{2\Im(z) }|m(z+1)+2n|^2)+\kappa\exp(-\alpha\frac{\pi }{\Im(z) }|mz+n|^2)
\end{equation}
which is the energy functional associated to the special Gaussian type completely monotone function $\exp(-\alpha x^2)$ with a free parameter $\alpha>0$.

The celebrated Bernstein theorem (\cite{Bernstein}) states that for any completely monotone function $F$, there is an non-negative finite Borel measure on $[0,\infty)$ with cumulative distribution function $\lambda(\alpha)$ such that
\begin{equation}\nonumber
F(x)=\int_0^\infty e^{-\alpha x} d\lambda(\alpha).
\end{equation}

As a consequence we obtain   the following integral representation of $\mathcal{E}_\kappa(z)$: there is an non-negative finite Borel measure on $[0,\infty)$ with cumulative distribution function $\lambda(\alpha)$ such that
\begin{equation}\label{Bern}
\mathcal{E}_\kappa(z)=\int_0^\infty \mathcal{W}_\kappa(z;\alpha) d\lambda(\alpha).
\end{equation}

The formula (\ref{Bern}) reduces the proof of Theorems \ref{Th1}-\ref{Th3} to the study of the universal  properties $ \mathcal{W}_\kappa (z;\alpha)$, which is the main contribution of this paper.

\subsection{Competing systems with Riesz potential}

One interesting and important competing potential is the Riesz potential type. We discuss the results on competing systems  with Riesz potential here.
Let $\mathcal{F}(x):=\frac{1}{x^s},s>1$ be the Riesz potential  which is a special case of completely monotone functions.
We consider the associated competing system energy functional
\begin{equation}\label{Riesz}
\mathcal{E}_R(z;\kappa,s):=\mathcal{E}_R(\kappa,s,\Lambda)=\sum_{\mathbb{P}\in\Lambda\backslash \{0\}} \Big((1-\kappa)\frac{1}{|\frac{\mathbb{P}+1}{2}|^{2s}}+\kappa \frac{1}{|\mathbb{P}|^{2s}}\Big), s>1,
\end{equation}
with $\Lambda =\sqrt{\frac{\pi}{\Im(z)}}\Big({\mathbb Z}\oplus z{\mathbb Z}\Big)$.

Observe that up to a constant (depending on $s$)  the Riesz sum  $\sum_{\mathbb{P}\in\Lambda\backslash \{0\}}  \frac{1}{|\mathbb{P}|^{2s}}$   is the  Eisenstein series $G(z,s)$ which is defined as
\begin{equation}\label{Gdef1}
G(z,s):=\frac{1}{2}\sum_{(m,n)\in\mathbb{Z}^2\backslash \{0\}}\frac{\big(\Im(z)\big)^s}{|mz+n|^{2s}}.
\end{equation}
See \cite{Cohen}.

The lattice energy functional $\mathcal{E}_R(z;\kappa,s)$ under Riesz potential (defined at (\ref{Riesz})) can be written up to a constant (depending on $s$) as
\begin{equation}\label{EisenW}
\mathbf{G}_R (z;\kappa, s):=(1-\kappa)\mathbf{G}(\frac{z+1}{2},s)+\kappa \mathbf{G}(z,s), \kappa\in[0,1],
\end{equation}
where $\mathbf{G} (z,s):= \pi^{-s} \Gamma (s) G(z,s)$ is a rescaled Eisenstein function which also appears frequently in number theory (\cite{Cohen}).

Theorems \ref{Th1}-\ref{Th3} can be applied to $ \mathbf{G}_R (z;\kappa,s)$. However an important observation, as shown in the appendix, is that
although the $s-$deriveatives are not completely monotone functions, we can still apply the techniques in proving Theorem \ref{Th1}-\ref{Th3} to the $s-$derivatives.

\begin{theorem}\label{Th4} Let  $\mathbf{G}_R (z; \kappa,s)$ be defined in \eqref{EisenW}.

\noindent
(a) For any $\kappa\in[0,1],s>1,j=1,2,3\cdots$, it follows that

\begin{equation}\nonumber
\min_{z\in \mathbb{H}}\frac{d^j}{ds^j}\mathbf{G}_R (z;\kappa,s)=\min_{z\in\Omega_e}\frac{d^j}{ds^j}\mathbf{G}_L(z;\kappa,s).
\end{equation}

\noindent
(b) There exists $\kappa_1<\kappa_2\in(0,1)$ depending on  $s$ such that $ [\frac{1}{3},\frac{2}{3}]\subseteq[\kappa_1, \kappa_2], \kappa_1+\kappa_2=1$ and
\begin{itemize}
  \item Case A: $\kappa\in[0,\kappa_1]$, $\min_{z\in \mathbb{H}}\frac{d^j}{ds^j}\mathbf{G}_R (z;\kappa,s)=\min_{z\in\Omega_{ea}}\frac{d^j}{ds^j}\mathbf{G}_R(z;\kappa,s)$; the minimizer corresponds to rectangular lattice in this case;
  \item Case B: $\kappa\in[\kappa_1,\kappa_2]$, $Minimizer_{z\in \mathbb{H}}\frac{d^j}{ds^j}\mathbf{G}_R (z;\kappa,s)=i$; the minimizer corresponds to square lattice in this case;
  \item Case C: $\kappa\in[\kappa_2,1]$, $\min_{z\in \mathbb{H}}\frac{d^j}{ds^j}\mathbf{G}_R (z;\kappa,s)=\min_{z\in\Omega_{eb}}\frac{d^j}{ds^j}\mathbf{G}_R (z;\kappa,s)$; the minimizer corresponds to rhombic lattice in this case;
\end{itemize}

In particular, when $\kappa=1$, $Minimizer_{z\in \mathbb{H}}\frac{d^j}{ds^j}\mathbf{G} (z,s)=\frac{1}{2}+i\frac{\sqrt3}{2}$, which corresponds to hexagonal lattice; when $\kappa=0$, $Minimizer_{z\in \mathbb{H}}\frac{d^j}{ds^j}\mathbf{G} (\frac{z+1}{2}, s)=\sqrt{3}i$, which corresponds to $\sqrt3:1$ rectangular lattice.

\end{theorem}

\begin{remark} The Eisentein series $G(z,s)$ is closely related to the Epstein zeta function $\zeta_Q(s)$, that is defined for
$\Re(s)>1$ by
\begin{equation}\nonumber
\zeta_Q(s):=\frac{1}{2}\sum_{(m,n)\in\mathbb{Z}^2\backslash \{0\}}\frac{1}{Q(m,n)^s},
\end{equation}
where $Q(x,y)=ax^2+bxy+cy^2$ is a positive definite quadratic form with real coefficients$($i.e., $a>0$, $c>0$ and $b^2-4ac<0$$)$.

Denote that $z=\frac{-b}{2a}+i\frac{\sqrt{4ac-b^2}}{2a}$ and $D=b^2-4ac$. Then the Epstein zeta function and Eisenstein series is connected by
\begin{equation}\label{Zeta}
\zeta_Q(s)=(\frac{|D|}{4})^{-\frac{s}{2}} G(z,s).
\end{equation}
See \cite{Cohen}. We also have an analogue theorem for Epstein zeta function.

\end{remark}

%%%%%%%%%%%%%%%%%%%%%%%%%%%%%%%%%

By the Kronecker limit formula(See for example  \cite{Siegel} and \cite{Serfaty2012} for the details), one has
\begin{equation}\aligned
\frac{d}{ds}\mid_{s=1}G(z,s)
=-\frac{1}{2}\log \big | \im \big (z \big ) \eta\big (z \big) \big|-\log(2\pi).
\endaligned\end{equation}
It follows that
\begin{equation}\aligned\label{YYY}
&\frac{d}{ds}\mid_{s=1}G_R (z;\kappa,s),
{G}_R (z;\kappa, s):=(1-\kappa){G}(\frac{z+1}{2},s)+\kappa {G}(z,s)\\
=&-\frac{1}{2}\Big(\kappa \log \big | \im \big (z \big ) \eta\big (z \big) \big|
  + (1-\kappa)\log \big | \im \big(\frac{z+1}{2}\big)
  \eta \big(\frac{z+1}{2}\big) \big |\Big)-\log(2\pi),
\endaligned\end{equation}
where
\begin{equation}
  \label{eta-intro}
  \eta(z) = e^{\frac{\pi}{3} z i}\prod_{n=1}^\infty\big(1- e^{2\pi nz i}\big)^4
\end{equation}
is the fourth power of the Dedekind eta function.
The derivative formula \eqref{YYY} connects the lattice energy functional with Riesz potential to lattice energy functional with Coulomb potential.
For convenience, define
\begin{equation}\label{fb-intro}
  \mathcal{N}_\kappa (z):= -\Big(\kappa \log \big | \im \big (z \big ) \eta\big (z \big) \big|
  + (1-\kappa)\log \big | \im \big(\frac{z+1}{2}\big)
  \eta \big(\frac{z+1}{2}\big) \big |\Big).
\end{equation}
The functional $\mathcal{N}_\kappa (z)$ is of the form of Dedekind eta functions, is a type of lattice energy functional with
Coulomb potential.

The reduced energy functional $ \mathcal{N}_\kappa (z)$ arises in the study of Triblock Ohta-Kawasaki system. For $\kappa=1$ it arises in the Diblock copolymer Ohta-Kawasaki system and also Ginzburg-Landau theory. See \cite{LRW} and \cite{Serfaty2012}.
Applying Theorem \ref{Th4} we then recover the results in our previous paper \cite{LRW}, in which we proved these results by a different method.
We state them here since our new results here include these as a special case and we have proved these results without using any PDE theory. This also confirms that {\bf the hexagonal-rhombic-square-rectangular lattice phases transition} in Theorem \ref{Th3} is quite general in various periodical competing
system.

In the numerical computation done in \cite{LRW}, we have for $\mathcal{N}_\kappa (z)$
$$
\kappa_1:= 0.1867\dots,\; \kappa_2:=0.8132\cdots.
$$
This confirms that the upper and lower bounds in Theorem \ref{Th3}, in which we proved that $
\kappa_1\leq\frac{1}{3},\; \kappa_2\geq\frac{2}{3}.$

\medskip

The paper is organized as follows: In section \ref{PreA}, we aim to give some basic properties of the lattice energy functional and
provide some basic tools in estimates. We prove Theorem \ref{Th1} and Theorem \ref{Th2}, Theorem \ref{Th3} in Section \ref{Theorem1} and Section \ref{Theorem2}
respectively.

\section{Preliminaries: Invariances of $\mathcal{E}_\kappa (z)$ and estimates on theta functions  }\label{PreA}
\setcounter{equation}{0}

Let $\mathcal{E}_\kappa (z)$ be the lattice energy functional defined at (\ref{F1}). The end point cases of $\mathcal{E}_\kappa(z)$ are
 \begin{equation}\label{E0}
\mathcal{E}_0(z)=
\sum_{(m,n)\in\mathbb{Z}^2\backslash \{0\}}\mathcal{F}(\frac{\pi }{2\Im(z) }|m(z+1)+2n|^2)
\end{equation}
and
 \begin{equation}\label{E1}
\mathcal{E}_1(z)=
\sum_{(m,n)\in\mathbb{Z}^2\backslash \{0\}}\mathcal{F}(\frac{\pi }{\Im(z) }|mz+n|^2).
\end{equation}
Then $\mathcal{E}_\kappa$ can be written as interpolation of $\mathcal{E}_0$ and $ \mathcal{E}_1$:
 \begin{equation}\label{Ekappa}
\mathcal{E}_\kappa(z)=\kappa\mathcal{E}_1(z)+(1-\kappa)\mathcal{E}_0(z).
\end{equation}

 In this section, we collect some basic properties of the functional $\mathcal{E}_\kappa$, including invariance and monotone properties, and some key estimates of $\mathcal{W}_\kappa$.

We first study the invariance properties of $ \mathcal{E}_\kappa$. We use the following definition of fundamental domain which is slightly different from the classical definition (see \cite{Mon1988}):
\begin{definition} [page 108, \cite{Eva1973}]\label{DefF}
The fundamental domain associated to group $G$ is a connected domain $\mathcal{D}$ satisfying
\begin{itemize}
  \item For any $z\in\mathbb{H}$, there exists an element $\pi\in G$ such that $\pi(z)\in\overline{\mathcal{D}}$;
  \item Suppose $z_1,z_2\in\mathcal{D}$ and $\pi(z_1)=z_2$ for some $\pi\in G$, then $z_1=z_2$ and $\pi=\pm Id$.
\end{itemize}
\end{definition}

For example, let $\Gamma$ be the group generated by the actions $\tau\mapsto -\frac{1}{\tau},\;\; \tau\mapsto \tau+1$, then
by Definition \ref{DefF}, the fundamental domain associated to modular group $\Gamma$ is
\begin{equation}\aligned\label{Fd1}
\mathcal{D}_\Gamma:=\{
z\in\mathbb{H}: |z|>1,\; -\frac{1}{2}<x<\frac{1}{2}
\}.
\endaligned\end{equation}
which is open.  Note that the fundamental domain can be open. (See [page 30, \cite{Apo1976}].)

Next we introduce another two groups related  to the functional $\mathcal{E}_\kappa(z)$. The generators of these groups are given by
\begin{equation}\aligned\label{GroupG1}
\mathcal{G}_1: \hbox{the group generated by} \;\;\tau\mapsto -\frac{1}{\tau},\;\; \tau\mapsto \tau+1,\;\;\tau\mapsto -\overline{\tau},
\endaligned\end{equation}
\begin{equation}\aligned\label{GroupG2}
\mathcal{G}_2: \hbox{the group generated by} \;\;\tau\mapsto -\frac{1}{\tau},\;\; \tau\mapsto \tau+2,\;\;\tau\mapsto -\overline{\tau}.
\endaligned\end{equation}
It is easy to see that
the fundamental domains associated to group $\mathcal{G}_j,j=1,2$ denoted by $\mathcal{D}_{\mathcal{G}_1},\mathcal{D}_{\mathcal{G}_2}$ are
\begin{equation}\aligned\label{Fd3}
\mathcal{D}_{\mathcal{G}_1}:=\{
z\in\mathbb{H}: |z|>1,\; 0<x<\frac{1}{2}
\}
\endaligned\end{equation}
\begin{equation}\aligned\label{Fd4}
\mathcal{D}_{\mathcal{G}_2}:=\{
z\in\mathbb{H}: |z|>1,\; 0<x<1
\}.
\endaligned\end{equation}
Clearly we have that $\mathcal{G}_2 \subseteq \mathcal{G}_1,\;\;\mathcal{D}_{\mathcal{G}_1}\subseteq\mathcal{D}_{\mathcal{G}_2}$.

%%%%%%%%%%%%%%%%%%%%%%%%%%%%%%%%%%%%%%%%%%%%%%%%%%%%%%%%%%%%%%%%%%

The following two Lemmas  can be checked directly from the definition.
\begin{lemma}\label{LG1} We have the following fundamental  invariant properties of the functionals:
\begin{itemize}
  \item $\mathcal{E}_1(z)$ is invariant under the group $\mathcal{G}_1$, i.e., for any $\gamma\in\mathcal{G}_1$
  $$
  \mathcal{E}_1(\gamma(z))=\mathcal{E}_1(z).
  $$
  \item $\mathcal{E}_0(z)$ is invariant under the group $\mathcal{G}_2$, i.e.,
 for any $\pi\in\mathcal{G}_2$
  $$
  \mathcal{E}_0(\pi(z))=\mathcal{E}_0(z).
  $$
  \item $\mathcal{E}_\kappa(z)$$(\kappa\in(0,1))$ is invariant under the group $\mathcal{G}_2$, i.e.,
 for any $\pi\in\mathcal{G}_2$
  $$
  \mathcal{E}_0(\pi(z))=\mathcal{E}_0(z).
  $$

\end{itemize}

As a consequence, the fundamental domain of the functionals $\mathcal{E}_1(z)$ and $\mathcal{E}_0(z)$
is $\mathcal{D}_{\mathcal{G}_1}$ and $\mathcal{D}_{\mathcal{G}_2}$  respectively, and  the fundamental domain for the general function $\mathcal{E}_\kappa(z)$$(\kappa\in(0,1))$ is $\mathcal{D}_{\mathcal{G}_2}$.

\end{lemma}

%%%%%%%%%%%%%%%%%%%%%%%%%%%%%%%%%%%%%%%%%%%%%%%%%%%%%%%%%%%%%%%%%%%
By the integral representation formula (\ref{Bern})  and the monotonicity properties of $\mathcal{W}_0 $ and $\mathcal{W}_1$ (see \cite{LW}), we have the following  important monotonicity properties for  the functionals $\mathcal{E}_0(z), \mathcal{E}_1(z)$ in the corresponding domain:
\begin{theorem} \label{Lemma2} There holds
\begin{itemize}
  \item
\begin{equation}\aligned\nonumber
\frac{\partial}{\partial x}\mathcal{E}_0(z)>0, \;\;\forall \;z\in \mathcal{D}_{\mathcal{G}_2},
\endaligned\end{equation}
  \item
$$
\frac{\partial}{\partial x}\mathcal{E}_1(z)<0,\;\;\forall \;z\in\Omega_{\mathcal{C}_1}.
$$
Here
\begin{equation}\aligned\nonumber
\Omega_{\mathcal{C}_1}:=\{z\mid 0<x<\frac{1}{2}, y>\sqrt{x-x^2}
\}.
\endaligned\end{equation}

\end{itemize}

\end{theorem}
\begin{remark} Montgomery \cite{Mon1988} proved that
\begin{equation}\aligned\label{MonLLL}
\frac{\partial}{\partial x}\mathcal{E}_1(z)<0,\;\;\forall \;z\in \Omega_{\mathcal{C}_0}:=\{
z\in\mathbb{H}: y>\frac{1}{2},\; 0<x<\frac{1}{2}
\}.
\endaligned\end{equation}
Theorem \ref{Lemma2}  improves this result to a larger domain $\Omega_{\mathcal{C}_1}$ as
$\Omega_{\mathcal{C}_0} \subset \Omega_{\mathcal{C}_1}$. Furthermore, $\Omega_{\mathcal{C}_1}$ contains a corner at $z=0$, which makes the proof much more involved.
\end{remark}

 We state two corollaries related to the functional $\mathcal{E}_\kappa(z)$.

\begin{corollary}\label{Coro1}
$$
\frac{\partial}{\partial x}\mathcal{E}_1(z)>0,\;\forall z\in\Omega_{\mathcal{C}_2}.
$$
Here
\begin{equation}\aligned\nonumber
\Omega_{\mathcal{C}_2}:=\{z\mid \frac{1}{2}<x<1, y>\sqrt{x-x^2}
\}.
\endaligned\end{equation}
\end{corollary}

\begin{proof} Since $z\mapsto1-\overline{z}\in\mathcal{G}_1$, by Lemma \ref{LG1}, we have $\mathcal{E}_1(1-\overline{z})=\mathcal{E}_1(z)$.
Thus
\begin{equation}\aligned\label{RR}
\frac{\partial}{\partial x}\mathcal{E}_1(1-\overline{z})=-\frac{\partial}{\partial x}\mathcal{E}_1(z).
\endaligned\end{equation}
The result follows by Theorem \ref{Lemma2}.

\end{proof}

By Theorem \ref{Lemma2} and Corollary \ref{Coro1} we have

\begin{theorem}\label{Coro2}  For $\kappa\in[0,1]$,
$$
\frac{\partial}{\partial x}\mathcal{E}_\kappa(z)>0,\;\forall z\in{\mathcal{R}_R}, j=1,2.
$$
Here
$$
{\mathcal{R}_R}:=\Omega_{\mathcal{C}_2}\cap\mathcal{D}_{\mathcal{G}_2}=\{z\mid \frac{1}{2}<\Re(z)<1, |z|>1\}.
$$

\end{theorem}
By Theorem \ref{Coro2}, it remains to consider functional on the domain
\begin{equation}
\label{RL}
{\mathcal{R}_L}:=\{z\mid 0<\Re(z)<\frac{1}{2}, |z|>1\}.
\end{equation}
We shall solve this problem in Section \ref{Theorem1}.

%%%%%%%%%%%%%%%%%%%%%%%%%%%%%%%%%%%%%%%%%%%%%%%%%%%%%%%%%%%%%%%%%%%%%%%%%%%%%%%%%%%%%%%%%
%%%%%%%%%%%%%%%%%%%%%%%%%%%%%%%%%%%%%%%%%%%%%%%%%%%%%%%%%%%%%%%%%%%%%%%%%%%%%%%%%%%%%%%%%
%%%%%%%%%%%%%%%%
\vskip0.2in

Before going to the proof, we also need some auxiliary functions from number theory.   We first recall the following well-known Jacobi triple product formula:
\begin{equation}\aligned\label{Jacob1}
\prod_{m=1}^\infty(1-x^{2m})(1+x^{2m-1}y^2)(1+\frac{x^{2m-1}}{y^2})=\sum_{n=-\infty}^\infty x^{n^2} y^{2n}
 \endaligned\end{equation}
for complex numbers $x,y$ with $|x|<1$, $y\neq0$.

The Jacob theta function is defined as
\begin{equation}\aligned\nonumber
\vartheta_J(z;\tau):=\sum_{n=-\infty}^\infty e^{i\pi n^2 \tau+2\pi i n z},
 \endaligned\end{equation}
and the classical one-dimensional theta function  is given by
\begin{equation}\aligned\label{TXY}
\vartheta(X;Y):=\vartheta_J(Y;iX)=\sum_{n=-\infty}^\infty e^{-\pi n^2 X} e^{2n\pi i Y}.
 \endaligned\end{equation}
Hence by the Jacobi triple product formula \eqref{Jacob1}, it holds
\begin{equation}\aligned\label{Product}
\vartheta(X;Y)=\prod_{n=1}^\infty(1-e^{-2\pi n X})(1+e^{-2(2n-1)\pi X}+2e^{-(2n-1)\pi X}\cos(2\pi Y)).
 \endaligned\end{equation}
By Poisson Summation Formula, one has
\begin{equation}\aligned\label{PPP2}
\vartheta(X;Y)=X^{-\frac{1}{2}}\sum_{n\in \mathbb{Z}} e^{-\frac{\pi(n-Y)^2}{X}}.
\endaligned\end{equation}

Thus the functional $\mathcal{W}_1(z;\alpha)$ can be written in terms of one-dimensional theta function as follows:

\begin{equation}\aligned\label{PPP3}
\mathcal{W}_1(z;\alpha)&=\sum_{(m,n)\in\mathbb{Z}^2} e^{-\alpha\pi\frac{1}{y }|nz+m|^2}
=\sum_{n\in\mathbb{Z}}e^{-\alpha \pi y n^2}\sum_{m\in\mathbb{Z}} e^{-\frac{\alpha \pi (nx+m)^2}{y}}\\
&=\sqrt{\frac{y}{\alpha}}\sum_{n\in\mathbb{Z}}e^{-\alpha \pi y n^2}\vartheta(\frac{y}{\alpha};-nx)=\sqrt{\frac{y}{\alpha}}\sum_{n\in\mathbb{Z}}e^{-\alpha \pi y n^2}\vartheta(\frac{y}{s};nx)\\
&=2\sqrt{\frac{y}{\alpha}}\sum_{n=1}^\infty e^{-\alpha \pi y n^2}\vartheta(\frac{y}{\alpha};nx).
\endaligned\end{equation}

We state  the following  very important properties of the theta function whose proof will be proved in the appendix.
We use these properties in estimates of the proof of Theorem \ref{Th1}.

\begin{lemma}\label{LemmaT3}
Let $Q(X;Y)=-\frac{1}{ \sin (2\pi Y)} \frac{\partial }{\partial Y}\vartheta(X;Y) $. Then there holds the following lower and upper bounds of $Q(X;Y)$
$$
\underline{\vartheta} (X)\leq Q(X;Y)\leq \overline{\vartheta} (X), \;\hbox{for  any}\; Y,
$$
where
\begin{equation}\label{vartheta1}
\underline{\vartheta} (t)=
\begin{cases}
\frac{1}{\max_{\beta\in[0,k]}{\sin(2\pi \beta)}}t^{-\frac{3}{2}}(e^{-\frac{\pi k^2}{t}}(k-e^{-\frac{\pi(1-2k)}{t}}(1-k))\\
+\sum_{m=1}^\infty e^{-\frac{\pi (k+m)^2}{t}}(k+m-e^{-\frac{\pi(2m+1)(1-2k)}{t}}(m+1-k)))\; (0<t\leq k<\frac{1}{2}),\\
(1-\mu(t))4\pi e^{-\pi t}\;(t>0),
\end{cases}
\end{equation}
and
\begin{equation}
\label{vartheta2}
\overline{\vartheta} (t)=
\begin{cases}
t^{-\frac{3}{2}}\sum_{m=-\infty}^\infty e^{-\frac{\pi m^2}{t}}(1-\frac{2\pi m^2}{t})\; (0<t\leq k <\frac{\pi}{2}),\\
(1+\mu(t))4\pi e^{-\pi t}\;(t>0).
\end{cases}
\end{equation}
Here $k$ and $a$ are positive numbers and
\begin{equation}\label{mmmx}
\mu(t):=\sum_{n=2}^\infty n^2 e^{-\pi(n^2-1)t}.
\end{equation}
 (For $t$ small, we can use
$$
Q(t;k)=\frac{1}{\sin(2\pi k)}t^{-\frac{3}{2}}(e^{-\frac{\pi k^2}{t}}(k-e^{-\frac{\pi(1-2k)}{t}}(1-k))+\sum_{m=1}^\infty e^{-\frac{\pi (k+m)^2}{t}}(k+m-e^{-\frac{\pi(2m+1)(1-2k)}{t}}(m+1-k))).)
$$

\end{lemma}

\section{Minimum principle: Proof of Theorem \ref{Th1}}\label{Theorem1}
\setcounter{equation}{0}

In this section, we prove Theorem \ref{Th1}. First   we observe that the functional $ \mathcal{W}_\kappa (z;\alpha)$ satisfies  the following duality principle: For all $b \in \mathbb{R}$,
\[ \mathcal{W}_b(w) = \mathcal{W}_{1-b} (z), \ \  z \in \mathbb{H} \ \ \mbox{and} \ \
 w=\frac{z-1}{z+1} \in \mathbb{H}. \]

 By the duality principle(see more details in Lemma \ref{l-dual}) and  the Bernstein representation formula (\ref{Bern}),
we only need to consider the case $ \kappa \in [0, \frac{1}{2}]$.

Next by the transition formula
 \begin{equation}\label{W12}
 \aligned
\mathcal{W}_{\kappa}(z;\alpha)=\kappa \mathcal{W}_{\frac{1}{2}}(z;\alpha)+(1-2\kappa)\mathcal{W}_{0}(z;\alpha)
 \endaligned\end{equation}
 we only need to study $ \mathcal{W}_{\frac{1}{2}} $ and $ \mathcal{W}_{0} $.  Furthermore since by Fourier transform
\begin{equation}\aligned\label{W124}
\mathcal{W}_{\kappa}(z;\alpha)=\alpha\mathcal{W}_{\kappa}(z;\frac{1}{\alpha})
 \endaligned\end{equation}
  we only need to consider $ \alpha \in [1, +\infty)$.  By Theorem \ref{Coro2}  we only need to consider $ z \in {\mathcal{R}_L}$ (defined at   (\ref{RL})). We also use the formula in \eqref{PPP3},
  \begin{equation}\aligned\nonumber
\mathcal{W}_1(z;\alpha)
=2\sqrt{\frac{y}{\alpha}}\sum_{n=1}^\infty e^{-\alpha \pi y n^2}\vartheta(\frac{y}{\alpha};nx).
\endaligned\end{equation}
Then we use the bounds in Lemma \ref{LemmaT3} to prove the estimates.

The following is the main result we shall prove in this section:

\begin{theorem}\label{LemmaT1} On the half fundamental domain bounded by $x=0,x=\frac{1}{2}$ and $x^2+y^2=1$, we have
 \begin{equation}\aligned\label{W123}
\frac{\partial}{\partial x}\mathcal{W}_{\frac{1}{2}}(z;\alpha)\geq0, \forall z\in \overline{ \mathcal{R}_{L}}
 \endaligned\end{equation}
where $\mathcal{R}_L$ is defined at (\ref{RL}). If the equality holds, then $x=0$.
\end{theorem}

By Corollary 4.2 of \cite{LW}, we also have $\frac{\partial}{\partial x}\mathcal{W}_{0}(z;\alpha)\geq0, \forall z\in \overline{ \mathcal{R}_{L}}$. Theorem \ref{Th1} then follows from
 (\ref{W123}) and Theorem \ref{Coro2}.

The proof of Theorem \ref{LemmaT1} follows from  the following two lemmas which establish the lower bound of $\frac{\partial}{\partial x}\mathcal{W}_{\frac{1}{2}}(z;\alpha)$ in  different ranges of $x$.

%%%%%%%%%%%%%%%%%%%%%%%%%%%%%%%%%%%%%%%%%%%%%%%%%%%%%%%%%%%%%%%%%%%%%%%%%%%%%%%%%%%%%%%%%%%%%%%%%%%%%%
\begin{lemma}\label{HHHA}  For $x \in [\frac{1}{4},\frac{1}{2}]$ we have the following  lower bound estimates for $\frac{\partial}{\partial x}\mathcal{W}_{\frac{1}{2}}(z;\alpha)$:
  \begin{equation}\aligned\label{HHHAA}
&\frac{\partial}{\partial x}\mathcal{W}_{\frac{1}{2}}(z;\alpha)\\
\geq&\sqrt{\frac{y}{\alpha}}e^{-\frac{\pi\alpha y}{2}}\sin(\pi x)
\Big(
\frac{1}{2\sqrt2}\underline{\vartheta} (\frac{y}{2\alpha})-\sqrt2(1+\epsilon_1(\alpha))e^{-\frac{\pi \alpha y}{2}}\overline{\vartheta} (\frac{y}{\alpha})
 -(1+\epsilon_2(\alpha))e^{-\frac{3\pi \alpha y}{2}}\overline{\vartheta} (\frac{y}{2\alpha})
\Big),
 \endaligned\end{equation}
 where $\epsilon_1(\alpha) <3.18\cdot 10^{-9},\;\epsilon_2(\alpha) < 2.782\cdot10^{-6}$ and $ \overline{\vartheta}, \underline{\vartheta}$ are defined in (\ref{vartheta1})-(\ref{vartheta2}).

 Furthermore, for $ y \geq \frac{ \sqrt{3}}{2}, \alpha \geq 1$, there holds
 \begin{equation}\aligned\label{HardAA}
\frac{1}{2\sqrt2}\underline\vartheta(\frac{y}{2\alpha})-\sqrt2(1+\epsilon_1(\alpha))e^{-\frac{\pi \alpha y}{2}}\overline\vartheta(\frac{y}{\alpha})
 -(1+\epsilon_2(\alpha))e^{-\frac{3\pi \alpha y}{2}}\overline\vartheta(\frac{y}{2\alpha})\geq0.
\endaligned\end{equation}
\end{lemma}
%%%%%%%%%%%%%%%%%%%%%%%%%%%%%%%%%%%%%%%%%%%%%%%%%%%%%%%%%%%%%%%%%%%%%%%%%%%%%%%%%%%%%%%%%%%%%%%%%%%%%

\begin{lemma}\label{HHHB} For $ x\in [0, \frac{1}{4}]$ we have the following lower bound estimates for $\frac{\partial}{\partial x}\mathcal{W}_{\frac{1}{2}}(z;\alpha)$:
\begin{equation}\aligned\label{HHHBB}
&\frac{\partial}{\partial x}\mathcal{W}_{\frac{1}{2}}(z;\alpha)\\
\geq&\sqrt{\frac{y}{\alpha}}e^{-\frac{\pi\alpha y}{2}}\sin(\pi x)
\Big(
\frac{1}{2\sqrt2}\underline\vartheta(\frac{y}{2\alpha})
-2(
1+\epsilon_3 (\alpha))e^{-\frac{\pi\alpha y}{2}}\overline\vartheta(\frac{y}{\alpha})
-{\sqrt2}(
1+\epsilon_4 (\alpha)
)e^{-\frac{3\pi\alpha y}{2}}\overline\vartheta(\frac{y}{2\alpha})
\Big)
 \endaligned\end{equation}
where $\epsilon_3 (\alpha) <1.1\cdot 10^{-4}$ and $\epsilon_4 (\alpha) <5.7 \cdot 10^{-16}.$

Furthermore, for $y\geq\frac{\sqrt{15}}{4}, \alpha\geq1$, we have
\begin{equation}\aligned\label{HardBB}
\frac{1}{2\sqrt2}\underline\vartheta(\frac{y}{2\alpha})
-2(
1+\epsilon_3 (\alpha))e^{-\frac{\pi\alpha y}{2}}\overline\vartheta(\frac{y}{\alpha})
-{\sqrt2}(
1+\epsilon_4 (\alpha)
)e^{-\frac{3\pi\alpha y}{2}}\overline\vartheta(\frac{y}{2\alpha})
>0.
 \endaligned\end{equation}

\end{lemma}

In the rest of this section we  shall prove Lemma \ref{HHHA} and Lemma \ref{HHHB}. By (\ref{W124}), we only need to consider the case $\alpha\geq1$. Let us start with the computation
\begin{equation}\aligned\label{J200}
\frac{\partial}{\partial x}\mathcal{W}_{\frac{1}{2}}(z;\alpha)=\sqrt{\frac{y}{\alpha}}\Big(&
\sum_{n=1}^\infty\frac{1}{2\sqrt2}ne^{-\frac{\pi \alpha yn^2}{2}}\frac{\partial}{\partial Y}\vartheta(\frac{y}{2\alpha};Y)\mid_{Y=n\frac{x+1}{2}}\\
\\&+\sum_{n=1}^\infty ne^{-\pi\alpha yn^2}\frac{\partial}{\partial Y}\vartheta(\frac{y}{\alpha};Y)\mid_{Y=nx}\Big).
 \endaligned\end{equation}

To refine the estimates, we divide the proof into two parts by $x\in[0,\frac{1}{4}]$ or $[\frac{1}{4},\frac{1}{2}]$. We first prove (\ref{HHHAA}) of Lemma \ref{HHHA}.

\noindent
{\bf Proof of (\ref{HHHAA}) of  Lemma \ref{HHHA}:}  Let $x\in[\frac{1}{4},\frac{1}{2}]$. We rewrite
\begin{equation}\aligned\label{Ialpha1}
\frac{\partial}{\partial x}
\mathcal{W}_{\frac{1}{2}}(z;\alpha)
=&\sqrt{\frac{y}{\alpha}}e^{-\frac{\pi \alpha y}{2}}\Big(
\frac{1}{2\sqrt2}\frac{\partial}{\partial Y}\vartheta(\frac{y}{2\alpha};Y)\mid_{Y=\frac{x+1}{2}}
+e^{-\frac{\pi \alpha y}{2}}\frac{\partial}{\partial Y}\vartheta(\frac{y}{\alpha};Y)\mid_{Y=x}\\
&\;\;\;\;\;\;\;\;\;\;+\frac{1}{\sqrt2}e^{-\frac{3\pi \alpha y}{2}}\frac{\partial}{\partial Y}\vartheta(\frac{y}{2\alpha};Y)\mid_{Y=x+1}\Big)\\
&+\sqrt{\frac{y}{\alpha}}\Big(
2e^{-4\pi \alpha y}\frac{\partial}{\partial Y}\vartheta(\frac{y}{\alpha};Y)\mid_{Y=2x}
+\sum_{n=3}^\infty\frac{1}{2\sqrt2}ne^{-\frac{\pi\alpha yn^2}{2}}\frac{\partial}{\partial Y}\vartheta(\frac{y}{2\alpha};Y)\mid_{Y=n\frac{x+1}{2}}\\
&\;\;\;\;\;\;\;\;\;\;+\sum_{n=3}^\infty ne^{-\pi\alpha y n^2}\frac{\partial}{\partial Y}\vartheta(\frac{y}{\alpha};Y)\mid_{Y=nx}\Big)\\
:=&\sqrt{\frac{y}{\alpha}}e^{-\frac{\pi \alpha y}{2}}\cdot J_1
+(\frac{y}{\alpha})^{\frac{1}{2}}\cdot J_2,
 \endaligned\end{equation}
where $J_1$ and $J_2$ are defined respectively at the last equality.

By the lower and upper bounds of $\vartheta$ in Lemma \ref{LemmaT3}, we have
\begin{equation}\aligned\label{J1}
 J_1
 &\geq\frac{1}{2\sqrt2}\underline{\vartheta} (\frac{y}{2\alpha})\sin(\pi x)-e^{-\frac{\pi \alpha y}{2}}\overline{\vartheta} (\frac{y}{\alpha})\sin(2\pi x)
 -\frac{1}{\sqrt2}e^{-\frac{3\pi \alpha y}{2}}\overline{\vartheta} (\frac{y}{2\alpha})\sin(2\pi x)\\
&=\sin(\pi x)\Big(
\frac{1}{2\sqrt2}\underline{\vartheta} (\frac{y}{2\alpha})-e^{-\frac{\pi \alpha y}{2}}\overline{\vartheta} (\frac{y}{\alpha})2\cos(\pi x)
 -\frac{1}{\sqrt2}e^{-\frac{3\pi \alpha y}{2}}\overline{\vartheta} (\frac{y}{2\alpha})2\cos(\pi x)
\Big)\\
&\geq\sin(\pi x)\Big(
\frac{1}{2\sqrt2}\underline{\vartheta} (\frac{y}{2\alpha})-\sqrt2e^{-\frac{\pi \alpha y}{2}}\overline{\vartheta} (\frac{y}{\alpha})
 -e^{-\frac{3\pi \alpha y}{2}}\overline{\vartheta} (\frac{y}{2\alpha})
\Big),
 \endaligned\end{equation}
 where $ \underline{\vartheta}$ and $ \overline{\vartheta}$  are defined at (\ref{vartheta1}) and (\ref{vartheta2}) respectively.

Here we have used the fact that $\sin (\pi (x+1))\leq0, \cos(\pi x)\leq\frac{\sqrt2}{2}\;\hbox{for}\;x\in[\frac{1}{4},\frac{1}{2}]$.

Next, we are going to estimate $J_2$.
By \eqref{Ialpha1}, using $|\sin(ny)|\leq n|\sin(y)|,y\in\R$ and Lemma \ref{LemmaT3}, we have
\begin{equation}\label{J214} \aligned
J_2 \geq
-\frac{1}{4\sqrt2}\overline{\vartheta} (\frac{y}{2\alpha})\sin(2\pi x)\sum_{n=3}^\infty n^2 e^{-\frac{\pi \alpha yn^2}{2}}
-\overline{\vartheta} (\frac{y}{\alpha})\sin(2\pi x)\sum_{n=3}^\infty n^2e^{-\pi \alpha y n^2}.
\endaligned
\end{equation}

Combining (\ref{J1}) and (\ref{J214}) we have
\begin{equation} \label{J124} \aligned
&J_1+e^{\frac{\pi\alpha y}{2}}J_2\geq \\
&
\Big(
\frac{1}{2\sqrt2}\underline{\vartheta} (\frac{y}{2\alpha})\sin(\pi x)-(1+\epsilon_1(\alpha))e^{-\frac{\pi \alpha y}{2}}\overline{\vartheta} (\frac{y}{\alpha})\sin(2\pi x)
 -(1+\epsilon_2(\alpha))\frac{1}{\sqrt2}e^{-\frac{3\pi \alpha y}{2}}\overline{\vartheta} (\frac{y}{2\alpha})\sin(2\pi x)
\Big)\\
&
\geq\sin(\pi x)\Big(
\frac{1}{2\sqrt2}\underline{\vartheta} (\frac{y}{2\alpha})-(1+\epsilon_1(\alpha))\sqrt2e^{-\frac{\pi \alpha y}{2}}\overline\vartheta(\frac{y}{\alpha})
 -(1+\epsilon_2(\alpha))e^{-\frac{3\pi \alpha y}{2}}\overline{\vartheta} (\frac{y}{2\alpha})
\Big), \;\hbox{for}\;x\in[\frac{1}{4},\frac{1}{2}],
\endaligned
\end{equation}
where the constants $\epsilon_1(\alpha)$ and $\epsilon_2(\alpha)$ are defined by
\begin{equation}
\epsilon_1(\alpha):=\sum_{n=3}^\infty n^2 e^{-\pi\alpha y(n^2-1)}\leq\sum_{n=3}^\infty n^2 e^{\frac{-\pi\sqrt3(n^2-1)}{2}}
\end{equation}
and
\begin{equation} \epsilon_2(\alpha):=\frac{1}{4}\sum_{n=3}^\infty n^2 e^{-\frac{\pi \alpha y(n^2-4)}{2}}
\leq \frac{1}{4}\sum_{n=3}^\infty n^2 e^{-\frac{\pi \sqrt3 (n^2-4)}{2}}
\end{equation}
It is easy to see that $ \epsilon_1 <3.18\cdot 10^{-9}, \epsilon_2 <2.782\cdot10^{-6}$.

The inequality (\ref{HHHAA}) then  follows from \eqref{Ialpha1} and (\ref{J124}).

\qed

\medskip

\noindent
{\bf Proof of (\ref{HardAA}) of Lemma \ref{HHHA}:} In the following we prove that  on the half fundamental domain with $x\in[\frac{1}{4},\frac{1}{2}]$ (equivalently, $y\geq\frac{\sqrt3}{2}$),  (\ref{HardAA}) holds  by assuming $ \epsilon_1=\epsilon_2=0$. The proof can be easily modified when two small constants are added.   We split the proof into three subcases as follows
$$\mathbf{A}: \frac{y}{\alpha},\frac{y}{2\alpha}\in(0,k];\mathbf{B}: \frac{y}{\alpha},\frac{y}{2\alpha}\in(k,+\infty);\mathbf{C}: \frac{y}{2\alpha}\in(0,k],\frac{y}{\alpha}\in(k,+\infty),$$
where the parameter $k\in(0,\frac{1}{2}]$ is to be chosen later.

\vskip0.1in

\noindent
{\bf Case $\mathbf{A}: \frac{y}{\alpha},\frac{y}{2\alpha}\in(0,k]$.} By Lemma \ref{LemmaT3}, we have
\begin{equation}\aligned\label{J1A}
 &\frac{1}{2\sqrt2}\underline\vartheta(\frac{y}{2\alpha})-\sqrt2e^{-\frac{\pi \alpha y}{2}}\overline\vartheta(\frac{y}{\alpha})
 -e^{-\frac{3\pi \alpha y}{2}}\overline\vartheta(\frac{y}{2\alpha})\\
\geq& \frac{1}{2\sqrt2}(\frac{y}{2\alpha})^{-\frac{3}{2}}\frac{1}{\sin(2\pi k)}
\Big(
e^{-\pi k^2\frac{2\alpha}{y}}k-e^{-\pi(1-k)^2\frac{2\alpha}{y}}(1-k)\\
&+e^{-\pi(1+k)^2\frac{2\alpha}{y}}(1+k-e^{-3\pi(1-2k)\frac{2\alpha}{y}})
\Big)
-\sqrt2 e^{-\frac{\pi \alpha y}{2}}(\frac{y}{\alpha})^{-\frac{3}{2}}
\Big(
1+(2-16\pi\frac{\alpha}{y})e^{-4\pi\frac{\alpha}{y}}
\Big)\\
&-e^{-\frac{3\pi \alpha y}{2}}(\frac{y}{2\alpha})^{-\frac{3}{2}}
\Big(
1+(2-16\pi\frac{2\alpha}{y})e^{-4\pi\frac{2\alpha}{y}})
\Big)\\
&\geq(\frac{y}{2\alpha})^{-\frac{3}{2}}
\Big(
\frac{1}{2\sqrt2}\frac{1}{\sin(2\pi k)} e^{-\pi k^2\frac{2\alpha}{y}}(k-\varepsilon_k)
-2e^{-\frac{\pi\alpha y}{2}}-e^{-\frac{3\pi \alpha y}{2}}\Big)\\
&\geq(\frac{y}{2\alpha})^{-\frac{3}{2}}
\Big(
\frac{1}{2\sqrt2} e^{-\pi k^2\frac{2\alpha}{y}}(k-\varepsilon_k)
-2e^{-\frac{\pi\alpha y}{2}}-e^{-\frac{3\pi \alpha y}{2}}\Big).
\endaligned\end{equation}
Here $\varepsilon_k$ is determined by
\begin{equation}\aligned\nonumber
\varepsilon_k:=\max_{\frac{y}{\alpha}\in[0,k]}\frac{e^{-\pi(1-k)^2\frac{2\alpha}{y}}(1-k)}{e^{-\pi k^2\frac{2\alpha}{y}}k} \leq\frac{1-k}{k}e^{-\pi(1-2k)\frac{2}{k}}.
 \endaligned\end{equation}
 Let $ k \in [0, 0.3]$. Then it is easy to see that
$\varepsilon_k\leq\varepsilon_{0.3}=0.000536\cdots<\frac{3}{5000}$
which yields that
\begin{equation}\aligned\label{J1A1}
 &\frac{1}{2\sqrt2}\underline\vartheta(\frac{y}{2\alpha})-\sqrt2e^{-\frac{\pi \alpha y}{2}}\overline\vartheta(\frac{y}{\alpha})
 -e^{-\frac{3\pi \alpha y}{2}}\overline\vartheta(\frac{y}{2\alpha})\\
 \geq&(\frac{y}{2\alpha})^{-\frac{3}{2}}
 \Big(
\frac{1}{2\sqrt2} e^{-\pi k^2\frac{2\alpha}{y}}(k-\frac{3}{5000})
-2e^{-\frac{\pi\alpha y}{2}}-e^{-\frac{3\pi \alpha y}{2}}\Big)\\
=& e^{-\pi k^2\frac{2\alpha}{y}}(\frac{y}{2\alpha})^{-\frac{3}{2}}
\Big(
\frac{k-\frac{3}{5000}}{2\sqrt2}-2e^{-\pi\alpha\frac{y^2-4k^2}{2y}}
-e^{-(\pi\alpha\frac{y^2-4k^2}{2y}+\pi\alpha y)}
\Big)\\
\geq&e^{-\pi k^2\frac{2\alpha}{y}}(\frac{y}{2\alpha})^{-\frac{3}{2}}
\Big(
\frac{k-\frac{3}{5000}}{2\sqrt2}-2(1+\frac{1}{5000})e^{-\pi\alpha\frac{y^2-4k^2}{2y}}
\Big)\\
\geq&e^{-\pi k^2\frac{2\alpha}{y}}(\frac{y}{2\alpha})^{-\frac{3}{2}}
\Big(
\frac{k-\frac{3}{5000}}{2\sqrt2}-2(1+\frac{1}{5000})e^{-\pi\frac{y^2-4k^2}{2k}}
\Big), \hbox{since}\;\alpha\geq\frac{y}{k}\\
\geq&e^{-\pi k^2\frac{2\alpha}{y}}(\frac{y}{2\alpha})^{-\frac{3}{2}}
\Big(
\frac{k-\frac{3}{5000}}{2\sqrt2}-2(1+\frac{1}{5000})e^{-\pi(\frac{3}{8k}-2k)}
\Big), \hbox{since}\;y^2\geq1-x^2=\frac{3}{4}.
 \endaligned\end{equation}
Here we have used the fact that
$$\max_{\frac{y}{\alpha}\in[0,k],k\in[0,0.3]}\frac{e^{-\pi\alpha y}}{2}
\leq\frac{e^{-\pi\frac{1}{0.3}y^2}}{2}\leq\frac{e^{-\pi\frac{1}{0.3}\frac{3}{4}}}{2}=0.00019\cdots<\frac{1}{5000}, \hbox{since}\;y^2\geq1-x^2\geq\frac{3}{4}.$$
An elementary computation shows that the function
\begin{equation}\label{fk}
f(k):=\frac{k-\frac{3}{5000}}{2\sqrt2}-2(1+\frac{1}{5000})e^{-\pi(\frac{3}{8k}-2k)} >0
\end{equation}
for $ k \in [0, k_f]$ where $ k_f= 0.25094\cdots$.

Therefore, by \eqref{J1A1},
\begin{equation}\aligned\label{J1A1}
  \frac{1}{2\sqrt2}\underline\vartheta(\frac{y}{2\alpha})-\sqrt2e^{-\frac{\pi \alpha y}{2}}\overline\vartheta(\frac{y}{\alpha})
 -e^{-\frac{3\pi \alpha y}{2}}\overline\vartheta(\frac{y}{2\alpha})
 \geq
 e^{-\pi k^2\frac{2\alpha}{y}}(\frac{y}{2\alpha})^{-\frac{3}{2}} f(k)>0
 \endaligned\end{equation}
if $k\in(0,k_f]$, where $k_f\approx0.25094\cdots$. This gives the first range of the free parameter $k$.
\vskip0.1in

\noindent
{\bf Case $\mathbf{B}: \frac{y}{\alpha},\frac{y}{2\alpha}\in(k,\infty)$.} By Lemma \ref{LemmaT3} again, we have
\begin{equation}\aligned\label{J1B}
&\frac{1}{2\sqrt2}\underline\vartheta(\frac{y}{2\alpha})-\sqrt2e^{-\frac{\pi \alpha y}{2}}\overline\vartheta(\frac{y}{\alpha})
 -e^{-\frac{3\pi \alpha y}{2}}\overline\vartheta(\frac{y}{2\alpha})\\
\geq&\frac{1}{2\sqrt2}(1-\mu(k))4\pi e^{-\pi\frac{y}{2\alpha}}
-\sqrt2 e^{-\frac{\pi\alpha y}{2}}(1+\mu(k))4\pi e^{-\pi\frac{y}{\alpha}}
-e^{-\frac{3\pi\alpha y}{2}}(1+\mu(k))4\pi e^{-\pi\frac{y}{2\alpha}}\\
=&4\pi e^{-\pi\frac{y}{2\alpha}}\Big(
\frac{1}{2\sqrt2}(1-\mu(k))-\sqrt2(1+\mu(k))e^{-\frac{\pi y}{2}(\alpha+\frac{1}{\alpha})}
-(1+\mu(k))e^{-\frac{3\pi\alpha y}{2}}
\Big)\\
\geq&4\pi e^{-\pi\frac{y}{2\alpha}}\Big(
\frac{1}{2\sqrt2}(1-\mu(k))-(\sqrt2+\delta)(1+\mu(k))e^{-\frac{\pi y}{2}(\alpha+\frac{1}{\alpha})}
\Big)\\
\geq&4\pi e^{-\pi\frac{y}{2\alpha}}\Big(
\frac{1}{2\sqrt2}(1-\mu(k))-(\sqrt2+e^{-\frac{\pi\sqrt3}{4}})(1+\mu(k))e^{-{\pi y}}\Big),\;\hbox{since}\;\frac{1}{2}(\alpha+\frac{1}{\alpha})\geq1, \alpha>0\\
\geq&4\pi e^{-\pi\frac{y}{2\alpha}}\Big(
\frac{1}{2\sqrt2}(1-\mu(k))-(\sqrt2+e^{-\frac{\pi\sqrt3}{4}})(1+\mu(k))e^{-\frac{\sqrt3\pi}{2}}\Big),\;\hbox{since}\;y\geq \sqrt{1-x^2}\geq\frac{\sqrt3}{2}.
 \endaligned\end{equation}
Here $\delta$ is defined by
$$
\delta=\max_{\alpha\geq1,y\geq\sqrt{1-x^2},x\in[\frac{1}{4},\frac{1}{2}]}\frac{e^{-\frac{3\pi\alpha y}{2}}}{e^{-\frac{\pi y}{2}(\alpha+\frac{1}{\alpha})}}.
$$
Then $$
\delta=\max_{\alpha\geq1,y\geq\sqrt{1-x^2},x\in[\frac{1}{4},\frac{1}{2}]} e^{-\pi y(\alpha-\frac{1}{2\alpha})}
\leq e^{-\pi y(1-\frac{1}{2})}\leq e^{-\frac{\pi\sqrt3}{4}}.
$$
Let
\begin{equation}\aligned\label{gk}
 g(k):&=\frac{1}{2\sqrt2}(1-\mu(k))-(\sqrt2+e^{-\frac{\pi\sqrt3}{4}})e^{-\frac{\sqrt3\pi}{2}}(1+\mu(k)), \;k\in[0,\frac{1}{2}]\\
 &=(\frac{1}{2\sqrt2}+(\sqrt2+e^{-\frac{\pi\sqrt3}{4}})e^{-\frac{\sqrt3\pi}{2}})
 (\frac{\frac{1}{2\sqrt2}-(\sqrt2+e^{-\frac{\pi\sqrt3}{4}})e^{-\frac{\sqrt3\pi}{2}}}{\frac{1}{2\sqrt2}+(\sqrt2+e^{-\frac{\pi\sqrt3}{4}})e^{-\frac{\sqrt3\pi}{2}}}
 -\mu(k)
 ).
 \endaligned\end{equation}
Then $g'(k)>0$ for any $k>0$ by $\mu'(k)>0$ and $g(k)>0$  if and only if $k>k_g$. Numerically, $k_g=0.22247\cdots$.

 Therefore by \eqref{J1}, \eqref{J1B} and \eqref{gk}, one has
\begin{equation}\aligned\nonumber
&\frac{1}{2\sqrt2}\underline\vartheta(\frac{y}{2\alpha})-\sqrt2e^{-\frac{\pi \alpha y}{2}}\overline\vartheta(\frac{y}{\alpha})
 -e^{-\frac{3\pi \alpha y}{2}}\overline\vartheta(\frac{y}{2\alpha})\geq4\pi e^{-\pi\frac{y}{2\alpha}}g(k)>0 \;\hbox{if}\;k>k_g=0.22247\cdots.
 \endaligned\end{equation}

This gives the second range of $k$ to be selected.
\vskip0.1in

\noindent
{\bf Case $\mathbf{C}: \frac{y}{2\alpha}\in(0,k],\frac{y}{\alpha}\in(k,\infty)$.} By Lemma \ref{LemmaT3}
\begin{equation}\aligned\label{J1C}
 &\frac{1}{2\sqrt2}\underline\vartheta(\frac{y}{2\alpha})-\sqrt2e^{-\frac{\pi \alpha y}{2}}\overline\vartheta(\frac{y}{\alpha})
 -e^{-\frac{3\pi \alpha y}{2}}\overline\vartheta(\frac{y}{2\alpha})\\
\geq&\frac{1}{2\sqrt2}(\frac{y}{2\alpha})^{-\frac{3}{2}}\frac{1}{\max_{x\in[0,k]}{\sin(2\pi x)}}\Big(
e^{-\pi k^2\frac{2\alpha}{y}}k-e^{-\pi(1-k)^2\frac{2\alpha}{y}}(1-k)\\
&\;\;\;\;+e^{-\pi(1+k)^2\frac{2\alpha}{y}}(1+k-e^{-3\pi(1-2k)\frac{2\alpha}{y}})
\Big)\\
&\;\;\;\;-\sqrt2 e^{-\frac{\pi \alpha y}{2}}(1+\mu(k))4\pi e^{-\pi\frac{y}{\alpha}}
-e^{-\frac{3\pi\alpha y}{2}}(\frac{y}{2\alpha})^{-\frac{3}{2}}(1+(2-4\pi\frac{2\alpha}{y})e^{-\pi\frac{2\alpha}{y}})\\
\geq&\frac{1}{2\sqrt2}(\frac{y}{2\alpha})^{-\frac{3}{2}}
e^{-\pi k^2\frac{2\alpha}{y}}(k-\epsilon_k)
-\sqrt2 e^{-\frac{\pi \alpha y}{2}}(1+\mu(k))4\pi e^{-\pi\frac{y}{\alpha}}
-e^{-\frac{3\pi\alpha y}{2}}(\frac{y}{2\alpha})^{-\frac{3}{2}},
\endaligned\end{equation}
where $\epsilon_k$ is defined by
\begin{equation}\label{ek1}
\epsilon_k:=\frac{e^{-\pi(1-k)^2\frac{2\alpha}{y}}(1-k)}{e^{-\pi k^2\frac{2\alpha}{y}}k}.
\end{equation}
Since the function $\epsilon_k$ will also appear in the proof of the next lemma, we single out the elementary but useful property  here:
\begin{equation}\label{ek2}
\frac{\partial}{\partial k}\epsilon_k=(1-k)(\pi-\frac{k}{1-k})\frac{e^{-\frac{\pi(1-2k)}{k}}}{k^3}.
\end{equation}
Hence $\epsilon_k$ is decreasing with respect to $k$ if $k\in(0,\frac{\pi}{1+\pi}]$.

By cases $\mathbf{A}$ and $\mathbf{B}$, we will select the parameter by $k\in(k_g,k_f)=(0.22247\cdots,0.25094\cdots)$.
\vskip0.1in
Let
$$
k\in(k_g,k_f).
$$
We further take $k\in(k_a,k_b]$, where $k_a=k_g$ and $k_b=\min\{k_f,\frac{1}{4}\}$. The number $\frac{1}{4}$ is taken by the bound to control that
\begin{equation}\aligned\nonumber
\frac{1}{\max_{x\in[0,k]}{\sin(2\pi x)}}\geq1
\endaligned\end{equation}
provided that $k\in(0,k_b]$. We keep the symbols $k_a, k_b$ to make the estimates  clearer, and keep in mind that
$k_a=0.22247\cdots,k_b=\frac{1}{4}$.

On the other hand it is easy to see that
$\epsilon_k<\epsilon_{k_f}=0.005843\cdots<\frac{3}{500}$. Note that
 $\frac{y}{2\alpha}\in(0,k]$ implies that
\begin{equation}\aligned\label{alpha1}
\alpha\geq\frac{y}{2k}.
\endaligned\end{equation}
Our method here works for all the $k$ belongs to the range $(k_a,k_b]$. Here for simplicity, we set
\begin{equation}\aligned\nonumber
k=\frac{1}{4}.
\endaligned\end{equation}
In this case, by \eqref{alpha1},
\begin{equation}\aligned\label{alpha1a}
\alpha\geq\sqrt3.
\endaligned\end{equation}
Here we have used that $y\geq\frac{\sqrt3}{2}$ implied by $x\in[\frac{1}{4},\frac{1}{2}]$
since $z=(x,y)$ belongs to the domain $\overline{{\mathcal{R}_L}}$.
Now to deal with the last inequality in \eqref{J1C}, we rewrite
\begin{equation}\aligned\label{alpha1b}
\alpha y=\alpha^2 \frac{y}{\alpha}\geq 3\frac{y}{\alpha}.
\endaligned\end{equation}

The bounds of $\frac{y}{\alpha}$ in this case implies that $\frac{1}{k}\leq\frac{2\alpha}{y}\leq\frac{2}{k}$.
Fixing all the parameters except $k$ in the expression $e^{-\pi k^2\frac{2\alpha}{y}}(k-\epsilon_k)$, then we have
\begin{equation}\aligned\label{bigk}
\frac{\partial}{\partial k}\big(e^{-\pi k^2\frac{2\alpha}{y}}(k-\epsilon_k)\big)=-e^{-\pi k^2\frac{2\alpha}{y}}
\big(
\frac{2\alpha}{y}\pi(k-\epsilon_k)-1
\big)
\endaligned\end{equation}
which is positive since trivially
\begin{equation}\aligned\nonumber
\frac{2\alpha}{y}\pi(k-\epsilon_k)-1\geq \pi(1-\frac{\epsilon_k}{k})-1>\frac{\pi}{2}-1>0.
\endaligned\end{equation}

To simplify the notations, we set $\gamma=\frac{y}{\alpha}$, hence $
\gamma\in(k,2k]=(\frac{1}{4},\frac{1}{2}].$

In view of \eqref{J1C}, by the monotonicity in \eqref{bigk},
\begin{equation}\aligned\label{J1C1}
 &\frac{1}{2\sqrt2}\underline\vartheta(\frac{y}{2\alpha})-\sqrt2e^{-\frac{\pi \alpha y}{2}}\overline\vartheta(\frac{y}{\alpha})
 -e^{-\frac{3\pi \alpha y}{2}}\overline\vartheta(\frac{y}{2\alpha})\\
\geq&
\gamma^{-\frac{3}{2}}e^{-\frac{\pi}{2}\frac{1}{\gamma}}(\frac{1}{4}-\epsilon_k)
-\sqrt2 e^{-\frac{\pi}{2}\alpha^2 \gamma}(1+\mu(\frac{1}{4}))4\pi e^{-\pi \gamma}
-e^{-\frac{3\pi}{2}\alpha^2 \gamma}(\frac{\gamma}{2})^{-\frac{3}{2}}.
\endaligned\end{equation}
Using the elementary but useful bound for $\alpha$ in \eqref{alpha1a}, we  further get the lower bound with the only variable $\gamma$, i.e.,
\begin{equation}\aligned\label{J1C1new}
 &\frac{1}{2\sqrt2}\underline\vartheta(\frac{y}{2\alpha})-\sqrt2e^{-\frac{\pi \alpha y}{2}}\overline\vartheta(\frac{y}{\alpha})
 -e^{-\frac{3\pi \alpha y}{2}}\overline\vartheta(\frac{y}{2\alpha})\\
\geq&
\gamma^{-\frac{3}{2}}e^{-\frac{\pi}{2}\frac{1}{\gamma}}(\frac{1}{4}-\epsilon_k)
-\sqrt2 e^{-\frac{3\pi}{2} \gamma}(1+\mu(\frac{1}{4}))4\pi e^{-\pi \gamma}
-e^{-\frac{9\pi}{2} \gamma}(\frac{\gamma}{2})^{-\frac{3}{2}}.
\endaligned\end{equation}

Now the only variable in above is the parameter $\gamma$, to make the structure clear,
let
\begin{equation}\aligned\label{h}
h(\gamma):=C_1\gamma^{-\frac{3}{2}}e^{-\frac{\pi}{2}\frac{1}{\gamma}}
-C_2e^{-\frac{5\pi}{2} \gamma}
-C_3e^{-\frac{9\pi}{2} \gamma}(\frac{\gamma}{2})^{-\frac{3}{2}}
\endaligned\end{equation}
with the constants determined by \eqref{J1C1new}
\begin{equation}\aligned\label{hc}
C_1&=\frac{1}{4}-\frac{3}{500}=0.244\\
C_2&=\sqrt2 (1+\mu(\frac{1}{4}))4\pi=24.80995446\cdots\\
C_3&=1,
\endaligned\end{equation}
where $\mu(X)$ is defined in \eqref{mmmx} and $\mu(\frac{1}{4})=0.39605042\cdots$.

By \eqref{J1C1new},
\begin{equation}\aligned\label{J1C1a}
 \frac{1}{2\sqrt2}\underline\vartheta(\frac{y}{2\alpha})-\sqrt2e^{-\frac{\pi \alpha y}{2}}\overline\vartheta(\frac{y}{\alpha})
 -e^{-\frac{3\pi \alpha y}{2}}\overline\vartheta(\frac{y}{2\alpha})
\geq h(\gamma), \gamma=\frac{y}{\alpha}.
\endaligned\end{equation}

 It remains to show that
\begin{equation}\aligned\label{hgammaw}
h(\gamma)>0,\hbox{for}\; \gamma\in(\frac{1}{4},\frac{1}{2}].
\endaligned\end{equation}
Simple calculation shows that
\begin{equation}\aligned\label{hgamma}
\big(\gamma^{-\frac{3}{2}}e^{-\frac{\pi}{2}\frac{1}{\gamma}}\big)'
=\gamma^{-\frac{7}{2}}e^{-\frac{\pi}{2}\frac{1}{\gamma}}(\frac{\pi}{2}-\frac{3}{2}\gamma^2),
\endaligned\end{equation}
while $\gamma\in(\frac{1}{4},\frac{1}{2}]$ in \eqref{hgammaw}, one has $\frac{\pi}{2}-\frac{3}{2}\gamma^2>0$. Hence, we infer that
$h(\gamma)$ is monotonically increasing for $\gamma\in(\frac{1}{4},\frac{1}{2}]$ and  it follows that
\begin{equation}\aligned\label{hgamma}
h(\gamma)>h(\frac{1}{4})=0.003734\cdots>0,\;\; \gamma\in(\frac{1}{4},\frac{1}{2}].
\endaligned\end{equation}
Therefore, by \eqref{J1C1new} and \eqref{h}, we have
\begin{equation}\aligned\label{KeyE}
 \frac{1}{2\sqrt2}\underline\vartheta(\frac{y}{2\alpha})-\sqrt2e^{-\frac{\pi \alpha y}{2}}\overline\vartheta(\frac{y}{\alpha})
 -e^{-\frac{3\pi \alpha y}{2}}\overline\vartheta(\frac{y}{2\alpha})
>0.
\endaligned\end{equation}

 By the same proof, we can take the parameter $k$ to be $\frac{1}{4}$ and  then all the estimates work. In fact, using the same method above, for any $k\in(k_g,\frac{1}{4}]$, we can show that the inequality \eqref{KeyE} holds.

Notice that the constants $\epsilon_j(\alpha),j=1,2$ are very close to zero. With almost the same computations in the lower bound estimates \eqref{KeyE}, we have that
for $x\in[\frac{1}{4},\frac{1}{2}]$ and $y\geq\frac{\sqrt3}{2}$
$$
\frac{1}{2\sqrt2}\underline\vartheta(\frac{y}{2\alpha})-(1+\epsilon_1(\alpha))\sqrt2e^{-\frac{\pi \alpha y}{2}}\overline\vartheta(\frac{y}{\alpha})
 -(1+\epsilon_2(\alpha))e^{-\frac{3\pi \alpha y}{2}}\overline\vartheta(\frac{y}{2\alpha})>0.
$$
The proof of (\ref{HardAA}) is thus complete.

\qed

Finally we prove Lemma \ref{HHHB}. The proof is similar to that of Lemma \ref{HHHA}. We first have

%%%%%%%%%%%%%%%%%%%%%%%%%%%%%%%%%%%%%%%%%%%%%%%%%%%%%%%%%%%%%%%%%%%%%%%%%%%%%%%%%%%%%

\noindent
{\bf Proof of (\ref{HHHBB}) of Lemma \ref{HHHB}:} Let $x\in[0,\frac{1}{4}].$
Similar to the proof of Lemma \ref{HHHA}, we divide
\begin{equation}\aligned\label{IJ2}
\frac{\partial}{\partial x}
&\mathcal{W}_{\frac{1}{2}}(z;\alpha)\\
=&e^{-\frac{\pi\alpha y}{2}}\sqrt{\frac{y}{\alpha}}\Big(
\frac{1}{2\sqrt2}\frac{\partial}{\partial Y}\vartheta(\frac{y}{2\alpha};Y)\mid_{Y=\frac{x+1}{2}}
+e^{-\frac{\pi\alpha y}{2}}\frac{\partial}{\partial Y}\vartheta(\frac{y}{\alpha};Y)\mid_{Y=x}\\
&\;\;\;+\frac{1}{\sqrt2}e^{-\frac{3\pi\alpha y}{2}}\frac{\partial}{\partial Y}\vartheta(\frac{y}{2\alpha};Y)\mid_{Y=x+1}
+2e^{-\frac{7\pi\alpha y}{2}}\frac{\partial}{\partial Y}\vartheta(\frac{y}{\alpha};Y)\mid_{Y=2x}
\Big)\\
&\;\;\;+\sqrt{\frac{y}{\alpha}}\Big(
\frac{3}{2\sqrt2}e^{-\frac{9\pi\alpha y}{2}}\frac{\partial}{\partial Y}\vartheta(\frac{y}{2\alpha};Y)\mid_{Y=3\frac{x+1}{2}}
+\sum_{n=4}^\infty\frac{1}{2\sqrt2}n e^{-\frac{\pi\alpha yn^2}{2}}\frac{\partial}{\partial Y}\vartheta(\frac{y}{2\alpha};Y)\mid_{Y=n\frac{x+1}{2}}\\
&\;\;\;+\sum_{n=3}^\infty ne^{-\pi\alpha y n^2}\frac{\partial}{\partial Y}\vartheta(\frac{y}{\alpha};Y)\mid_{Y=nx}
\Big)\\
& =\sqrt{\frac{y}{\alpha}}\Big(e^{-\frac{\pi\alpha y}{2}}J_{3}
+ J_{4}\Big).
 \endaligned\end{equation}
Here $J_{3} $ and $ J_{4}$ are defined at the last equality.

For the lower bound of $J_3$, since $ \sin (4\pi x) \geq0$, by Lemma \ref{LemmaT3}
\begin{equation}\aligned\label{J111}
J_3
&\geq
\frac{1}{2\sqrt2}\underline\vartheta(\frac{y}{2\alpha})\sin(\pi x)
-e^{-\frac{\pi\alpha y}{2}}\overline\vartheta(\frac{y}{\alpha})\sin(2\pi x)\\
&\;\;\;-\frac{1}{\sqrt2}e^{-\frac{3\pi\alpha y}{2}}\overline\vartheta(\frac{y}{2\alpha})\sin(2\pi x)
-2e^{-\frac{7\pi\alpha y}{2}}\overline\vartheta(\frac{y}{\alpha})\sin(4\pi x).
 \endaligned\end{equation}

Observe that for $x\in[0,\frac{1}{4}]$
$$
\frac{\partial}{\partial Y}\vartheta(\frac{y}{2\alpha};Y)\mid_{Y=3\frac{x+1}{2}}
\geq\underline\vartheta(\frac{y}{2\alpha})\sin(3\pi x)\geq0.
$$
To compare the terms in $J_3$, using $|\sin(ny)|\leq n|\sin(y)|,y\in\R$ and Lemma \ref{LemmaT3},
we estimate that
\begin{equation}\aligned\label{J122}
J_4
&\geq
-\frac{1}{4\sqrt2}\sum_{n=4}^\infty n^2e^{-\frac{\pi\alpha yn^2}{2}}\overline\vartheta(\frac{y}{2\alpha})\sin(2\pi x)
-\sum_{n=3}^\infty n^2e^{-\pi\alpha y n^2}\overline\vartheta(\frac{y}{\alpha})\sin(2\pi x).
 \endaligned\end{equation}

In fact, combining \eqref{J111} and \eqref{J122}, by collecting the terms, one has
\begin{equation}\aligned\label{PPP1}
J_3+e^{\frac{\pi\alpha y}{2}}J_4
&\geq
\frac{1}{2\sqrt2}\underline\vartheta(\frac{y}{2\alpha})\sin(\pi x)
-e^{-\frac{\pi\alpha y}{2}}\overline\vartheta(\frac{y}{\alpha})\sin(2\pi x)\Big(
1+4e^{-3\pi\alpha y}+\sum_{n=3}n^2 e^{-\pi\alpha y(n^2-1)}
\Big)\\
&\;\;\;-\frac{1}{\sqrt2}e^{-\frac{3\pi\alpha y}{2}}\overline\vartheta(\frac{y}{2\alpha})\sin(2\pi x)
\Big(
1+\frac{1}{4}\sum_{n=4}^\infty n^2 e^{-\frac{\pi\alpha y(n^2-4)}{2}}
\Big)\\
&\geq\sin(\pi x)\Big(
\frac{1}{2\sqrt2}\underline\vartheta(\frac{y}{2\alpha})
-2e^{-\frac{\pi\alpha y}{2}}\overline\vartheta(\frac{y}{\alpha})(
1+\epsilon_3 (\alpha))
-{\sqrt2}e^{-\frac{3\pi\alpha y}{2}}\overline\vartheta(\frac{y}{2\alpha})
(
1+\epsilon_4 (\alpha)
)
\Big).
 \endaligned\end{equation}
The constants $\epsilon_j(\alpha), j=3,4$ are defined by
$$
\epsilon_3 (\alpha)=4e^{-3\pi\alpha y}+\sum_{n=3}n^2 e^{-\pi\alpha y(n^2-1}, \
\epsilon_4 (\alpha)=\frac{1}{4}\sum_{n=4}^\infty n^2 e^{-\frac{\pi\alpha y(n^2-4)}{2}}.
$$

Since $x\in[0,\frac{1}{4}]$,  $y\geq\sqrt{1-x^2}\geq\frac{\sqrt{15}}{4}$. A direct calculation shows that $\epsilon_3 (\alpha) <1.1\cdot 10^{-4},\epsilon_4 (\alpha) < 5.7 \cdot 10^{-16}.$

Combining \eqref{IJ2} and \eqref{PPP1}, we obtain (\ref{HHHBB}).

\qed

\vskip0.3in

\noindent
{\bf Proof of (\ref{HardBB})  of Lemma \ref{HHHB}:} It suffices to prove (\ref{HardBB}) for $y\geq\sqrt{1-x^2}\geq\frac{\sqrt{15}}{4}$.

The proof is similar to that of (\ref{HHHAA}) of  Lemma \ref{HHHA}. We first divide the proof into three subcases
$$\mathbf{A}: \frac{y}{\alpha},\frac{y}{2\alpha}\in(0,k];\mathbf{B}: \frac{y}{\alpha},\frac{y}{2\alpha}\in(0,k]\in(k,+\infty);\mathbf{C}: \frac{y}{\alpha}\in(0,k],\frac{y}{2\alpha}\in(k,+\infty)$$
where $k$ will be chosen appropriately later.

\medskip

\noindent
{\bf Case $\mathbf{A}: \frac{y}{\alpha},\frac{y}{2\alpha}\in(0,k]$.} By Lemma \ref{LemmaT3}
\begin{equation}\aligned\nonumber
& \frac{1}{2\sqrt2}\underline\vartheta(\frac{y}{2\alpha})
-2(1+\epsilon_3 (\alpha))e^{-\frac{\pi\alpha y}{2}}\overline\vartheta(\frac{y}{\alpha})
-{\sqrt2}(1+\epsilon_4 (\alpha))e^{-\frac{3\pi\alpha y}{2}}\overline\vartheta(\frac{y}{2\alpha})\\
\geq&\frac{1}{2\sqrt2}(\frac{y}{2\alpha})^{-\frac{3}{2}}\frac{1}{\sin(2\pi k)}\big(
e^{-\pi k^2\frac{2\alpha}{y}}k-e^{-\pi(1-k)^2\frac{2\alpha}{y}}(1-k)
\big)\\
&-2(1+\epsilon_3 (\alpha))e^{-\frac{\pi\alpha y}{2}}(\frac{y}{\alpha})^{-\frac{3}{2}}
-\sqrt2(1+\epsilon_4 (\alpha))e^{-\frac{3\pi\alpha y}{2}}(\frac{y}{2\alpha})^{-\frac{3}{2}}\\
\geq&(\frac{y}{2\alpha})^{-\frac{3}{2}}e^{-\pi k^2\frac{2\alpha}{y}}
\Big(
\frac{1}{2\sqrt2}(k-\varepsilon_k)-e^{-\frac{\pi\alpha(y^2-4k^2)}{2y}}
2\sqrt2(1+\epsilon_3 (\alpha))(1+\frac{\sqrt2(1+\eta_2(\alpha))e^{-\pi\alpha y}}{2\sqrt2(1+\epsilon_3 (\alpha))})
\Big).
\endaligned\end{equation}
Here $\varepsilon_k$ is determined by
$$
\varepsilon_k:=\max_{\frac{y}{\alpha}\in[0,k]}\frac{e^{-\pi(1-k)^2\frac{2\alpha}{y}}(1-k)}{e^{-\pi k^2\frac{2\alpha}{y}}k}.
$$
Let
$$
c(\alpha):=2\sqrt2(1+\epsilon_3 (\alpha))(1+\frac{\sqrt2(1+\epsilon_4 (\alpha))e^{-\pi\alpha y}}{2\sqrt2(1+\eta_1(\alpha))})
$$
and
$$
F_0(k):=\frac{1}{2\sqrt2}(k-\varepsilon_k)-c(\alpha)e^{-\frac{\pi\alpha(y^2-4k^2)}{2y}}.
$$
Then
\begin{equation}\aligned\label{F0k}
 \frac{1}{2\sqrt2}\underline\vartheta(\frac{y}{2\alpha})
-2(1+\epsilon_3 (\alpha))e^{-\frac{\pi\alpha y}{2}}\overline\vartheta(\frac{y}{\alpha})
-{\sqrt2}(1+\epsilon_4(\alpha))e^{-\frac{3\pi\alpha y}{2}}\overline\vartheta(\frac{y}{2\alpha})
\geq(\frac{y}{2\alpha})^{-\frac{3}{2}}e^{-\pi k^2\frac{2\alpha}{y}}
F_0(k).
\endaligned\end{equation}

Let $k\in[0,0.3]$.
Since $\frac{y}{\alpha}\in[0,k]$ and $y\geq\frac{\sqrt{15}}{4}$, one has
$$
\varepsilon_k<\frac{3}{5000},\;c(\alpha)<2.8287
$$ and
\begin{equation}\aligned\label{AAAF}
F_0(k)\geq F(k):=\frac{1}{2\sqrt2}(k-\varepsilon_k)-c(\alpha)e^{-\pi(\frac{15}{32k}-2k)}.
\endaligned\end{equation}
Hence, by \eqref{F0k}, one has
\begin{equation}\aligned\label{TTTHHH}
 \frac{1}{2\sqrt2}\underline\vartheta(\frac{y}{2\alpha})
-2(1+\epsilon_3 (\alpha))e^{-\frac{\pi\alpha y}{2}}\overline\vartheta(\frac{y}{\alpha})
-{\sqrt2}(1+\epsilon_4 (\alpha))e^{-\frac{3\pi\alpha y}{2}}\overline\vartheta(\frac{y}{2\alpha})
\geq(\frac{y}{2\alpha})^{-\frac{3}{2}}e^{-\pi k^2\frac{2\alpha}{y}}
F(k).
\endaligned\end{equation}
Now $F(k)$ is an elementary function and is decreasing on $[0,0.3]$, and computation shows that $F(k)>0$ equivalents to
$$k\in(0,k_F), k_F=0.28689\cdots.$$
Therefore, by \eqref{TTTHHH}, one has
\begin{equation}\aligned\nonumber
 \frac{1}{2\sqrt2}\underline\vartheta(\frac{y}{2\alpha})
-2(1+\epsilon_3 (\alpha))e^{-\frac{\pi\alpha y}{2}}\overline\vartheta(\frac{y}{\alpha})
-{\sqrt2}(1+\epsilon_4 (\alpha))e^{-\frac{3\pi\alpha y}{2}}\overline\vartheta(\frac{y}{2\alpha})
\geq(\frac{y}{2\alpha})^{-\frac{3}{2}}e^{-\pi k^2\frac{2\alpha}{y}}
F(k)>0
\endaligned\end{equation}
provided that $k\in(0,k_F)$ with $k_F=0.28689\cdots$.

This gives the first range of $k$ to be selected in completing our proof.
\vskip0.1in

\noindent
{\bf Case $\mathbf{B}: \frac{y}{\alpha},\frac{y}{2\alpha}\in(k,\infty)$:}
By Lemma \ref{LemmaT3}
\begin{equation}\aligned\nonumber
& \frac{1}{2\sqrt2}\underline\vartheta(\frac{y}{2\alpha})
-2(1+\epsilon_3 (\alpha))e^{-\frac{\pi\alpha y}{2}}\overline\vartheta(\frac{y}{\alpha})
-{\sqrt2}(1+\epsilon_4 (\alpha))e^{-\frac{3\pi\alpha y}{2}}\overline\vartheta(\frac{y}{2\alpha})\\
\geq&\frac{1}{2\sqrt2}(1-\mu(k))4\pi e^{-\pi\frac{y}{2\alpha}}
-2(1+\epsilon_3 (\alpha))(1+\mu(k))4\pi e^{-\pi\frac{y}{\alpha}}e^{-\frac{\pi\alpha y}{2}}\\
&-\sqrt2(1+\epsilon_4 (\alpha))(1+\mu(k))4\pi e^{-\pi\frac{y}{2\alpha}}e^{-\frac{3\pi\alpha y}{2}}\\
=& 4\pi e^{-\frac{y}{2\alpha}}
\Big(
\frac{1}{2\sqrt2}(1-\mu(k))-(2(1+\epsilon_3 (\alpha)+\sqrt2(1+\epsilon_3 (\alpha))e^{-\pi y(\alpha-\frac{1}{2\alpha})}))e^{-\frac{\pi y}{2}(\alpha+\frac{1}{\alpha})}
(1+\mu(k))\Big).
\endaligned\end{equation}
Let
$$
d(\alpha):=(2(1+\epsilon_3 (\alpha)+\sqrt2(1+\epsilon_3 (\alpha))e^{-\pi y(\alpha-\frac{1}{2\alpha})}))e^{-\frac{\pi y}{2}(\alpha+\frac{1}{\alpha})}
$$
and

\begin{equation}\aligned\label{AAAG}
G(k):=\frac{\frac{1}{2\sqrt2}-d(\alpha)}{\frac{1}{2\sqrt2}+d(\alpha)}-\mu(k).
\endaligned\end{equation}
Then
\begin{equation}\aligned\nonumber
 \frac{1}{2\sqrt2}\underline\vartheta(\frac{y}{2\alpha})
-2(1+\epsilon_3 (\alpha))e^{-\frac{\pi\alpha y}{2}}\overline\vartheta(\frac{y}{\alpha})
-{\sqrt2}(1+\epsilon_4 (\alpha))e^{-\frac{3\pi\alpha y}{2}}\overline\vartheta(\frac{y}{2\alpha})
\geq\frac{4\pi e^{-\pi\frac{y}{2\alpha}}}{\frac{1}{2\sqrt2}+d(\alpha)}G(k).
\endaligned\end{equation}
Since $y\geq\frac{\sqrt{15}}{4}$, then $d(\alpha)\leq0.11025(\alpha\geq1)$ and $G'(k)=-\mu'(k)>0$. Hence numerically
$$G(k)>0\; \hbox{if}\; k>k_G:=0.22263\cdots.$$

 Therefore
\begin{equation}\aligned\nonumber
&\frac{1}{2\sqrt2}\underline\vartheta(\frac{y}{2\alpha})
-2(1+\epsilon_3(\alpha))e^{-\frac{\pi\alpha y}{2}}\overline\vartheta(\frac{y}{\alpha})
-{\sqrt2}(1+\epsilon_4(\alpha))e^{-\frac{3\pi\alpha y}{2}}\overline\vartheta(\frac{y}{2\alpha})\\
\geq&\frac{4\pi e^{-\pi\frac{y}{2\alpha}}}{\frac{1}{2\sqrt2}+d(\alpha)}G(k)
>0 \;\hbox{if}\;k>k_G.
\endaligned\end{equation}
This gives the second range of $k$ to be selected, i.e., $k>k_G(=0.22263\cdots)$.
\vskip0.1in

\noindent
{\bf Case $\mathbf{C}: \frac{y}{2\alpha}\in(0,k],\frac{y}{\alpha}\in(k,\infty)$.}
In this case, we first set $k\in(k_G,k_F)=(0.22263\cdots,0.28689\cdots)$ by the computations in subcases D and E above.
In using the comparison inequality later, we further set that
\begin{equation}\aligned\nonumber
k\in(k_A,k_B]
\endaligned\end{equation}
where
\begin{equation}\aligned\nonumber
k_A:=\max\{k_g, k_G\},\;\; k_B:=\min\{k_f, k_F, \frac{1}{4}\}.
\endaligned\end{equation}
Then, $k\in(k_1,k_2]=(0.22263\cdots,\frac{1}{4}]$. On the other hand, since $\frac{y}{2\alpha}\in(0,k]$, then $\alpha\geq\frac{y}{2k}$.
Again the bound $\frac{1}{4}$ is taken for controlling
\begin{equation}\aligned\nonumber
\frac{1}{\max_{x\in[0,k]}{\sin(2\pi x)}}\geq1
\endaligned\end{equation}
provided that $k\in(0,k_B]$.

By Lemma \ref{LemmaT3}
\begin{equation}\aligned\label{Feee1}
& \frac{1}{2\sqrt2}\underline\vartheta(\frac{y}{2\alpha})
-2(1+\epsilon_3 1(\alpha))e^{-\frac{\pi\alpha y}{2}}\overline\vartheta(\frac{y}{\alpha})
-{\sqrt2}(1+\epsilon_4 (\alpha))e^{-\frac{3\pi\alpha y}{2}}\overline\vartheta(\frac{y}{2\alpha})\\
\geq&\frac{1}{2\sqrt2}(\frac{y}{2\alpha})^{-\frac{3}{2}}\frac{1}{\max_{x\in[0,k]}{\sin(2\pi x)}}\big(
e^{-\pi k^2\frac{2\alpha}{y}}k-e^{-\pi(1-k)^2\frac{2\alpha}{y}}(1-k)
\big)\\
&\;\;\;\;-2(1+\epsilon_3 (\alpha)) e^{-\frac{\pi \alpha y}{2}}(1+\mu(k))4\pi e^{-\pi\frac{y}{\alpha}}
-{\sqrt2}(1+\epsilon_4 (\alpha))e^{-\frac{3\pi\alpha y}{2}}(\frac{y}{2\alpha})^{-\frac{3}{2}}\\
\geq&
\frac{1}{2\sqrt2}(k-\epsilon_k)(\frac{y}{2\alpha})^{-\frac{3}{2}}e^{-\pi k^2\frac{2\alpha}{y}}
-2(1+\epsilon_3 (\alpha))(1+\mu(k))4\pi e^{-\frac{\pi\alpha y}{2}} e^{-\pi\frac{y}{\alpha}}\\
&\;\;\;\;-\sqrt2(1+\epsilon_4 (\alpha))e^{-\frac{3\pi\alpha y}{2}}(\frac{y}{2\alpha})^{-\frac{3}{2}}\\
=&
(k-\epsilon_k)(\frac{y}{\alpha})^{-\frac{3}{2}}e^{-\pi k^2\frac{2\alpha}{y}}
-2(1+\epsilon_3 (\alpha))(1+\mu(k))4\pi e^{-\frac{\pi\alpha y}{2}} e^{-\pi\frac{y}{\alpha}}\\
&\;\;\;\;-\sqrt2(1+\epsilon_4 (\alpha))e^{-\frac{3\pi\alpha y}{2}}(\frac{y}{2\alpha})^{-\frac{3}{2}}.\\
\endaligned\end{equation}
Our method works for all $k\in(k_A,k_B]$, to make the structure clear, we set that
\begin{equation}\aligned\nonumber
k=k_B=\frac{1}{4}.
\endaligned\end{equation}
Since $\frac{y}{2\alpha}\in(0,k]$ and $y\geq\frac{\sqrt{15}}{4}$, one has
\begin{equation}\aligned\nonumber
\alpha\geq\frac{y}{2k}=2y\geq\frac{\sqrt{15}}{2}.
\endaligned\end{equation}
To simplify the lower bounds in \eqref{Feee1},  we use the following
\begin{equation}\aligned\nonumber
\alpha y=\alpha^2 \frac{y}{\alpha}\geq\frac{15}{4}\frac{y}{\alpha}.
\endaligned\end{equation}

Continuing from \eqref{Feee1}, one has
\begin{equation}\aligned\label{Feee2}
& \frac{1}{2\sqrt2}\underline\vartheta(\frac{y}{2\alpha})
-2(1+\epsilon_3(\alpha))e^{-\frac{\pi\alpha y}{2}}\overline\vartheta(\frac{y}{\alpha})
-{\sqrt2}(1+\epsilon_4(\alpha))e^{-\frac{3\pi\alpha y}{2}}\overline\vartheta(\frac{y}{2\alpha})\\
\geq&
(\frac{1}{4}-\frac{3}{500})(\frac{y}{\alpha})^{-\frac{3}{2}}e^{-\frac{\pi}{2}\frac{\alpha}{y}}
-2(1+\epsilon_3(\alpha))(1+\mu(\frac{1}{4}))4\pi  e^{-\frac{23}{8}\pi\frac{y}{\alpha}}\\
&\;\;\;\;-\sqrt2(1+\epsilon_4 (\alpha))e^{-\frac{45\pi\alpha y}{8}}(\frac{y}{2\alpha})^{-\frac{3}{2}}.
\endaligned\end{equation}
Let
\begin{equation}\aligned\nonumber
\gamma=\frac{y}{\alpha}.
\endaligned\end{equation}
Then in this case
\begin{equation}\aligned\nonumber
\gamma=\frac{y}{\alpha}\in(k,2k]=(\frac{1}{4},\frac{1}{2}].
\endaligned\end{equation}
Let
\begin{equation}\aligned\label{AAAH}
H(\gamma):=\mathcal{C}_1\gamma^{-\frac{3}{2}}e^{-\frac{\pi}{2}\frac{1}{\gamma}}
-\mathcal{C}_2 e^{-\frac{23}{8}\pi\gamma}
-\mathcal{C}_3 e^{-\frac{45\pi }{8}\gamma}(\frac{\gamma}{2})^{-\frac{3}{2}},
\endaligned\end{equation}
here
\begin{equation}\aligned\nonumber
\mathcal{C}_1&=\frac{1}{4}-\frac{3}{500}\\
\mathcal{C}_2&=2(1+\epsilon_3(\alpha))(1+\mu(\frac{1}{4}))4\pi\\
\mathcal{C}_3&=\sqrt2(1+\epsilon_4(\alpha)).
\endaligned\end{equation}

By \eqref{Feee2},
\begin{equation}\aligned\nonumber
\frac{1}{2\sqrt2}\underline\vartheta(\frac{y}{2\alpha})
-2(1+\epsilon_3 (\alpha))e^{-\frac{\pi\alpha y}{2}}\overline\vartheta(\frac{y}{\alpha})
-{\sqrt2}(1+\epsilon_4 (\alpha))e^{-\frac{3\pi\alpha y}{2}}\overline\vartheta(\frac{y}{2\alpha})
\geq H(\gamma), \gamma=\frac{y}{\alpha}.
\endaligned\end{equation}
Then, it suffices to prove that
\begin{equation}\aligned\nonumber
H(\gamma)>0,\;\; \gamma\in(\frac{1}{4},\frac{1}{2}].
\endaligned\end{equation}
The first part of the function $H(\gamma)$ was studied in \eqref{hgamma}, the second and third parts of the function are decreasing, thus $H(\gamma)$ is increasing on $(\frac{1}{4},\frac{1}{2}]$, hence
\begin{equation}\aligned\nonumber
H(\gamma)>H(\frac{1}{4}),\;\; \gamma\in(\frac{1}{4},\frac{1}{2}].
\endaligned\end{equation}
The numerical result shows that
\begin{equation}\aligned\nonumber
H(\frac{1}{4})>0.
\endaligned\end{equation}
Thus we complete the proof of case C.

In summary of cases A, B and C, we have
\begin{equation}\aligned\label{KeyF}
&\frac{1}{2\sqrt2}\underline\vartheta(\frac{y}{2\alpha})
-2(1+\epsilon_3(\alpha))e^{-\frac{\pi\alpha y}{2}}\overline\vartheta(\frac{y}{\alpha})
-{\sqrt2}(1+\epsilon_4(\alpha))e^{-\frac{3\pi\alpha y}{2}}\overline\vartheta(\frac{y}{2\alpha})\\
\geq& (\frac{y}{2\alpha})^{-\frac{3}{2}}e^{-\pi k^2\frac{2\alpha}{y}}
F(k)>0, \hbox{if}\;\frac{y}{\alpha}\in(0,k]\;, k\in(0,k_F),k_F=0.28689\cdots\\
\geq& \frac{4\pi e^{-\pi\frac{y}{2\alpha}}}{\frac{1}{2\sqrt2}+d(\alpha)}G(k)
>0 \;\hbox{if}\;k>k_G, k_G=0.22263\cdots\\
\geq& H(\gamma)>0,\gamma=\frac{y}{\alpha},\hbox{if}\;\frac{y}{2\alpha}\in(0,k],\frac{y}{\alpha}\in(k,\infty),k=\frac{1}{4},
\endaligned
\end{equation}
see the functions $F(k), G(k), H(\gamma)$ in \eqref{AAAF}, \eqref{AAAG} and \eqref{AAAH} respectively.
Therefore, by selecting $k=\frac{1}{4}$ in \eqref{KeyF}, one has
\begin{equation}\aligned\nonumber
\frac{1}{2\sqrt2}\underline\vartheta(\frac{y}{2\alpha})
-2(1+\epsilon_3 (\alpha))e^{-\frac{\pi\alpha y}{2}}\overline\vartheta(\frac{y}{\alpha})
-{\sqrt2}(1+\epsilon_4 (\alpha))e^{-\frac{3\pi\alpha y}{2}}\overline\vartheta(\frac{y}{2\alpha})
>0.
\endaligned
\end{equation}
The proof of (\ref{HardBB}) is complete.

\qed

\section{proof of Theorem \ref{Th2} and Theorem \ref{Th3}}\label{Theorem2}
\setcounter{equation}{0}

In this section, we study the monotonicity properties of $\mathcal{W}_{\kappa}(z;\alpha)$ on the imaginary axis $z=iy, y\geq 1$ and prove Theorems \ref{Th2} and \ref{Th3}.

%%%%%%%%%%%%%%%%%%%%%%%%%%%%%%%%%%%%%%%%%%%%%%%%%%%%%%%%%%%
Recall that from the definition of  $\mathcal{W}_{\kappa}(z;\alpha)$ we have
 \begin{equation}\aligned\nonumber
\mathcal{W}_{0}(iy;\alpha)&=\sum_{n,m}e^{-\frac{\pi \alpha}{2}\big((2n+1)^2y+\frac{(2m+1)^2}{y}\big)}+\sum_{n,m}e^{-2\pi \alpha(n^2y+\frac{m^2}{y})}\\
\mathcal{W}_{1}(iy;\alpha)&=\sum_{n,m}e^{-\pi \alpha(n^2y+\frac{m^2}{y})}.
 \endaligned\end{equation}

Here for convenience, throughout this section, we denote  $
\sum_{n,m}=\sum_{n=-\infty}^{n=\infty}\sum_{m=-\infty}^{m=\infty}.
$

In this section, we aim to establish that

\begin{theorem}\label{4th1} For $\sigma \in[0,2]$, the function $\mathcal{W}_{1}(iy;\alpha)+\sigma \mathcal{W}_{0}(iy;\alpha), y\in(0,\infty)$ achieves its unique minimum
at $y=1$.
\end{theorem}

As a consequence, by the Bernstein Theorem \cite{Bernstein}, we have
\begin{theorem}\label{4th0} For $\sigma \in[0,2]$, the function $\mathcal{E}_{1}(iy;\alpha)+\sigma \mathcal{E}_{0}(iy;\alpha), y\in(0,\infty)$ achieves its unique  minimum
at $y=1$.
\end{theorem}
\vskip0.1in

We first recall the duality property of $\mathcal{W}_\kappa$, which can be proved by Lemma \ref{LG1} and we omit the details here.
\begin{lemma}
\label{l-dual}
Under the transform $z \rightarrow w = \frac{z-1}{z+1}$ of $\mathbb{H}$,
\[ \mathcal{W}_1(z) = \mathcal{W}_0 (w) \ \ \mbox{and} \ \ \mathcal{W}_0(z) = \mathcal{W}_1(w), \ \ z \in \mathbb{H} \ \ \mbox{and} \ \
 w=\frac{z-1}{z+1} \in \mathbb{H}. \]
Consequently, for all $b \in \mathbb{R}$,
\[ \mathcal{W}_b(w) = \mathcal{W}_{1-b} (z), \ \  z \in \mathbb{H} \ \ \mbox{and} \ \
 w=\frac{z-1}{z+1} \in \mathbb{H}. \]

 More generally, if $h: z' \rightarrow w'= \frac{z'-1}{z'+1}$ and
 $g_1: z \rightarrow z'$, $g_2: w' \rightarrow w$ are transforms in ${\mathcal{G}_2}$,
 then
 \[ \mathcal{W}_b(w) = \mathcal{W}_{1-b} (z), \ \  z \in \mathbb{H} \ \ \mbox{and} \ \
 w=g_2 \circ h \circ g_1 (z) \in \mathbb{H}.\]
\end{lemma}

%%%%%%%%%%%%%%%%%%%%%%%%%%%%%%%%%%%%%%%%%%%%%%%%%%%%%%%%%%%%%%%%%%%%%%%%%%%
\vskip0.2in
Let us write $z=x+yi$ henceforth, and set
\[ X_b(z) = \frac{\partial \mathcal{W}_b(z)}{\partial x} = b X_1(z) + (1-b) X_0(z), \ \
Y_b(z) = \frac{\partial \mathcal{W}_b(z)}{\partial y}= b Y_1(z) + (1-b) Y_0(z). \]

The following two formulas relate
  $f_b$ on the upper half of the unit circle
  to $f_{1-b}$ on the upper half of the imaginary axis.

  \begin{lemma}
    \label{l-circle-im}
    Let the upper half of the unit circle be parametrized by
    $u+ i \sqrt{1-u^2}$, $u \in (-1,1)$. Then
    $\frac{\sqrt{1-u^2}}{1-u} i$ parametrizes the upper half of the
    imaginary axies, and
  \begin{align}
  \label{line-to-circle-X}
  X_b(u + i\sqrt{1-u^2}) &= \frac{\sqrt{1-u^2}}{1-u}  Y_{1-b}
  \Big(\frac{\sqrt{1-u^2}}{1-u} i \Big) \\  \label{line-to-circle-Y}
  Y_b(u + i \sqrt{1-u^2}) &= \frac{-u}{1-u}  Y_{1-b}
  \Big(\frac{\sqrt{1-u^2}}{1-u} i\Big)
  \end{align}
  hold for $u \in (-1,1)$.
    \end{lemma}
\begin{proof}
Consider the transform in Lemma \ref{l-dual}, $z \rightarrow w = \frac{z-1}{z+1}$.  With $z=x+yi$ and
$w=u+vi$,
\begin{equation}
  \label{key-transform}
  u = \frac{x^2+y^2-1}{(x+1)^2+y^2}, \ \ v=\frac{2y}{(x+1)^2+y^2}.
\end{equation}
Conversely,
\begin{equation}
  \label{key-transform-1}
  x = \frac{1-u^2-v^2}{(1-u)^2+v^2}, \ \ y=\frac{2v}{(1-u)^2+v^2}.
\end{equation}
Differentiate $f_b(w) = f_{1-b} (z)$ with respect to $u$ and $v$ to find
\begin{align*}
  &X_b(w) = X_{1-b}(z) \frac{\partial x}{\partial u} + Y_{1-b}(z)
  \frac{\partial y}{\partial u} \\
   & Y_b(w) = X_{1-b}(z) \frac{\partial x}{\partial v} + Y_{1-b}(z)
  \frac{\partial y}{\partial v}.
\end{align*}
When $w$ is on the unit circle, $z$ is on the imaginary axis. Since $\mathcal{W}_{1-b}$
is invariant under the reflection about the imaginary axis,
$X_{1-b}(z) =0$ on the imaginary axis. Also
\begin{equation}
   \frac{\partial y}{\partial u}\Big |_{|w|=1}  = \frac{\sqrt{1-u^2}}{1-u},
  \ \ \frac{\partial y}{\partial v}\Big |_{|w|=1}  = \frac{-u}{1-u}
\end{equation}
from which the lemma follows.
\end{proof}
%%%%%%%%%%%%%%%%%%%%%%%%%%%%%%%%%%%%%%%%%%%%%%%%%%%%%%%%%%%%%%%%%%%%%%%%%%%%%%%%%%%%%%%%%%%%%%%%%%%%%%%%%%%%%%%%%%%%%%%%%%%%%%%%%%%55

 %%%%%%%%%%%%%%%%%%%%%%%%%%%%%%%%%%%%%%%%%%%%%%%%%%%%%%%%%%%%%%
 By Lemma \ref{LG1}, one has
  \begin{equation}\aligned\label{JJJ1}
\mathcal{W}_{1}(i\frac{1}{y};\alpha)+\sigma \mathcal{W}_{0}(i\frac{1}{y};\alpha)=\mathcal{W}_{1}(iy;\alpha)+\sigma \mathcal{W}_{0}(iy;\alpha).
 \endaligned\end{equation}
Then to prove Theorem \ref{4th1}, we will show that
\begin{lemma}\label{4lemma11} Assume that $\sigma\in[0,2]$. Then
  \begin{equation}\aligned\label{lemma111}
\frac{d}{dy}\Big(\mathcal{W}_{1}(iy;\alpha)+\sigma \mathcal{W}_{0}(iy;\alpha)
\Big)\geq0,\;\; \hbox{for}\;\; y\geq1, \alpha \geq 1.
 \endaligned\end{equation}

 \end{lemma}

 The proof of (\ref{lemma111}) is through the operator $\mathcal{L}$
\begin{equation}
\label{Ldef}
\mathcal{L}Y:=Y''+\frac{2}{y}Y'= y^{-2} (y^2 Y^{'})^{'}, \ y>0
\end{equation}

Observe that if $\mathcal{L}f(y)\geq0$ for $y\geq 1>0$ and $f'(1)=0$, then $
f'(y)>0\;\;\hbox{for}\;\; y>1. $  From (\ref{JJJ1}) we see that $\frac{d}{dy} (\mathcal{W}_{1}(iy;\alpha)+\sigma \mathcal{W}_{0}(iy;\alpha))\bigl|_{y=1} =0$.  In view of this fact, to prove Lemma \ref{4lemma11}, it  suffices to prove that
 \begin{lemma}\label{4lemma1}Assume that $\sigma\in[0,2]$. Then
  \begin{equation}\aligned\nonumber
\mathcal{L}\Big(\mathcal{W}_{1}(iy;\alpha)+\sigma \mathcal{W}_{0}(iy;\alpha)
\Big)\geq0,\;\; \hbox{for}\;\; y\geq1, \alpha \geq 1.
 \endaligned\end{equation}

 \end{lemma}
Note that $\mathcal{L}\mathcal{W}_{1}(iy;\alpha)\geq$ for $\alpha>0, y\geq1$ is proved by Montgomery \cite{Mon1988} and stated in
Lemma \ref{ML}. To prove Lemma \ref{4lemma1}, it suffices to prove that it holds for $\sigma=2$.

To prove Lemma \ref{4lemma1} we compute first
\begin{lemma}\label{lemmaD}
One first has the following two identities by direct computation
 \begin{equation}\aligned\label{Iabb}
\mathcal{L}\Big(\sum_{n,m}e^{-\pi \alpha(n^2y+\frac{m^2}{y})}\Big)
&=I_a(y)-I_b(y),\\
\mathcal{L}\Big(\sum_{n,m}e^{-\frac{\pi \alpha}{2}\big((2n+1)^2y+\frac{(2m+1)^2}{y}\big)}\Big)
&=J_a(y)-J_b(y),
 \endaligned\end{equation}
where
 \begin{equation}\aligned \label{Iab}
I_{a}(y):&=\sum_{n,m}\pi^2\alpha^2(\frac{m^2}{y^2}-n^2)^2e^{-\pi \alpha(n^2y+\frac{m^2}{y})},\\
I_{b}(y):&=\sum_{n,m}\frac{2\pi \alpha n^2}{y}e^{-\pi \alpha(n^2y+\frac{m^2}{y})},\\
J_a(y):&=\sum_{n,m}\frac{\pi^2\alpha^2}{4}\big(\frac{(2m+1)^2}{y^2}-(2n+1)^2\big)^2e^{-\frac{\pi \alpha}{2}\big((2n+1)^2y+\frac{(2m+1)^2}{y}\big)},\\
J_b(y):&=\frac{\pi \alpha}{y}\sum_{n,m}(2n+1)^2e^{-\frac{\pi \alpha}{2}\big((2n+1)^2y+\frac{(2m+1)^2}{y}\big)}.
 \endaligned\end{equation}

\end{lemma}

 The following Lemma is due to Montgomery \cite{Mon1988}.
\begin{lemma}\label{ML}
\begin{itemize}
  \item $\mathcal{L}\Big(\sum_{n,m}e^{-\pi \alpha(n^2y+\frac{m^2}{y})}\Big)\geq0$ for $y\in[1,\infty)$ with equality holds only when $y=1$;
  \item $\mathcal{L}\Big(\sum_{n,m}e^{-\frac{\pi \alpha}{2}\big((2n+1)^2y+\frac{(2m+1)^2}{y}\big)}+\sum_{n,m}e^{-2\pi \alpha(n^2y+\frac{m^2}{y})}\Big)\geq0$
  for $y\in[\sqrt3,\infty)$ with equality holds only when $y=\sqrt3$.
\end{itemize}

\end{lemma}

The $\mathcal{L}\Big(\sum_{n,m}e^{-\frac{\pi \alpha}{2}\big((2n+1)^2y+\frac{(2m+1)^2}{y}\big)}\Big)$ will change sign on the interval $[1,\sqrt3]$, however, we establish the following

\begin{lemma}\label{Lemma43} At least for $\sigma=2$, it holds that
\begin{equation}\label{sigma}
\mathcal{L}\Big(\sum_{n,m}e^{-\pi \alpha(n^2y+\frac{m^2}{y})}+\sigma\sum_{n,m}e^{-\frac{\pi \alpha}{2}\big((2n+1)^2y+\frac{(2m+1)^2}{y}\big)}\Big)\geq0,  \hbox{for}\;\; y\in[1,\sqrt3].
\end{equation}
As a consequence, (\ref{sigma}) holds for $\sigma\in[0,2]$ as well.
\end{lemma}

The proof of Lemma \ref{Lemma43} is quite technique and we postpone it to the end of this section.
Define $\sigma_{M}>0$ as

\begin{equation}\nonumber\aligned
\sigma_{M}:=\max_{\sigma}\{
\sigma\mid: \mathcal{L}\Big(\sum_{n,m}e^{-\pi \alpha (n^2y+\frac{m^2}{y})}+\sigma\sum_{n,m}e^{-\frac{\pi \alpha }{2}\big((2n+1)^2y+\frac{(2m+1)^2}{y}\big)}\Big)\geq0,  \hbox{for}\;\; y\in[1,\sqrt3], \alpha\geq1
\}.
 \endaligned\end{equation}
Here, as the invariance for the parameter $\alpha$ in \eqref{JJJ1}, one only needs to consider the case $\alpha\geq1$.

We omit the proof of the following lemma that can be proved by the monotonicity property and Lemma \ref{Lemma43}.
\begin{lemma}\label{Lemma42} $\sigma_M$ satisfies
\begin{equation}\nonumber
\sigma_M\leq-\frac{\mathcal{L}'\Big(\sum_{n,m}e^{-\pi \alpha(n^2y+\frac{m^2}{y})}\Big)\mid_{y=1,\alpha=1}}
{\mathcal{L}'\Big(\sum_{n,m}e^{-\frac{\pi \alpha}{2}\big((2n+1)^2y+\frac{(2m+1)^2}{y}\big)}\Big)\mid_{y=1,\alpha=1}}.
\end{equation}

\end{lemma}
\begin{remark} The expression in Lemma \ref{Lemma42} is well-defined since one notes that
\begin{equation}\nonumber\aligned
&{\mathcal{L}\Big(\sum_{n,m}e^{-\pi \alpha(n^2y+\frac{m^2}{y})}\Big)\mid_{y=1}}=
{\mathcal{L}\Big(\sum_{n,m}e^{-\frac{\pi \alpha}{2}\big((2n+1)^2y+\frac{(2m+1)^2}{y}\big)}\Big)\mid_{y=1}}=0\\
&{\mathcal{L}'\Big(\sum_{n,m}e^{-\pi \alpha(n^2y+\frac{m^2}{y})}\Big)\mid_{y=1}}>0,
{\mathcal{L}\Big(\sum_{n,m}e^{-\frac{\pi \alpha}{2}\big((2n+1)^2y+\frac{(2m+1)^2}{y}\big)}\Big)\mid_{y=1}}<0.
 \endaligned\end{equation}
\end{remark}

\begin{theorem}\label{Th41} For any $\kappa\in[\frac{1}{1+\sigma_M},\frac{\sigma_M}{1+\sigma_M}]$,

\begin{equation}\nonumber
Minimizer_{z\in \mathbb{H}}\mathcal{E}_\kappa(z)=i,
\end{equation}
which is the purely imaginary number $i$, corresponding to the square lattice by our setting and is unique up the group $\mathcal{G}_2$ defined in \eqref{GroupG2}. In particular, via Lemma \ref{Lemma43},
$$
[\frac{1}{1+\sigma_M},\frac{\sigma_M}{1+\sigma_M}]\supseteq[\frac{1}{3}, \frac{2}{3}].
$$

\end{theorem}

%%%%%%%%%%%%%%%%%%%%%%%%%%%%%%%%%%%%%%%%%%%%%%%%%%%%%%%%%%%%%%%%%%%%%%%
\begin{proof}
Note that
\begin{equation}\label{GJJ}\aligned
&\mathcal{E}_{1}(z;\alpha)+\sigma \mathcal{E}_{0}(z;\alpha)\\
=&(1+\sigma)\Big(
\frac{1}{1+\sigma}\mathcal{E}_{1}(z;\alpha)+\frac{\sigma}{1+\sigma}\mathcal{E}_{0}(z;\alpha)
\Big)\\
=&(1+\sigma)\mathcal{E}_{\frac{1}{1+\sigma}}(z).
\endaligned\end{equation}

 By minimum principle in Theorem \ref{Th1}, one has
\begin{equation}\label{T1C}
\min_{z\in \mathbb{H}}\mathcal{E}_\kappa(z)=\min_{z\in\Omega_e}\mathcal{E}_\kappa(z),
\end{equation}
the minimizer may happen on $\mathbb{H}\backslash\Omega_e$, while it must be modular to some minimizer on $\Omega_e$
via the group $\mathcal{G}_2$.
By Theorem \ref{4th0} and \eqref{GJJ}, for $\kappa\in[\frac{1}{3},1]$ (or more generally $\kappa\in[\frac{1}{1+\sigma_M},1]$),
\begin{equation}\label{T1A}
Minimizer_{z\in\Omega_{ea}}\mathcal{E}_\kappa(z)=i.
\end{equation}
By Lemma \ref{l-circle-im} and \eqref{T1A}, for $\kappa\in[0,\frac{2}{3}]$ (or more generally $\kappa\in[0, \frac{\sigma_M}{1+\sigma_M}]$),
\begin{equation}\label{T1B}
Minimizer_{z\in\Omega_{eb}}\mathcal{E}_\kappa(z)=i.
\end{equation}
Therefore, it follows by \eqref{T1C}, \eqref{T1A} and \eqref{T1B}, for $\kappa\in[\frac{1}{1+\sigma_M},\frac{\sigma_M}{1+\sigma_M}]$,
\begin{equation}\label{T1AA}
Minimizer_{z\in\mathbb{H}}\mathcal{E}_\kappa(z)=i.
\end{equation}
The proof is thus complete.

\end{proof}

%%%%%%%%%%%%%%%%%%%%%%%%%%%%%%%%%
Next, we turn to the proof of Theorem \ref{Th3}.
\begin{proof}
By definition
\begin{equation}\nonumber
\mathcal{W}_\kappa(z;\alpha)=\kappa\mathcal{W}_1(z;\alpha)+(1-\kappa)\mathcal{W}_0(z;\alpha).
\end{equation}
To comparison, by deformation one has, for any $\kappa, \kappa_0(\neq0)$
\begin{equation}\label{TFFF}
\mathcal{W}_\kappa(z;\alpha)=\frac{\kappa}{\kappa_0}\mathcal{W}_{\kappa_0}(z;\alpha)+\frac{\kappa_0-\kappa}{\kappa_0}\mathcal{W}_0(z;\alpha).
\end{equation}
Define the least and largest number such that
\begin{equation}\nonumber
Minimizer_{z\in\mathbb{H}}\mathcal{E}_\kappa(z)=i
\end{equation}
is $\kappa_1$ and $\kappa_2$ respectively. It follows by duality Lemma \ref{l-dual} that
$$
\kappa_1+\kappa_2=1.
$$
Theorem \ref{Th41} asserts that
$$
\kappa_1\leq\frac{1}{3},\; \kappa_2\geq\frac{2}{3}.
$$
We also define that
$$
\kappa_1':=\frac{1}{1+\sigma_M},\; \kappa_2':=\frac{\sigma_M}{1+\sigma_M}.
$$
It follows by the definitions that
$$
\kappa_1'\geq \kappa_1,\; \kappa_2'\leq \kappa_2.
$$
Consider $\kappa\in[0,\kappa_1']$. By the definition of $\sigma_M$, one has
\begin{equation}\label{H1}
Minimizer_{z\in\Omega_{eb}}\mathcal{W}_{\kappa_1'}(z;\alpha)=i,
\end{equation}
and the minimizer is unique by the monotonicity induced in Lemma \ref{l-circle-im}, Lemma \ref{l-dual} and Lemma \ref{ML}.
On the other hand, via Lemma \ref{ML} and \ref{l-circle-im}, one has the functional $\mathcal{W}_{0}(z)$ is monotonically increasing on the circle $\Omega_{eb}$, i.e,
\begin{equation}\label{H2}
Minimizer_{z\in\Omega_{eb}}\mathcal{W}_{0}(z;\alpha)=i,
\end{equation}
and the minimizer is uniquely achieved.

From the comparison formula \eqref{TFFF}, we have
\begin{equation}\label{TFFFa}
\mathcal{W}_\kappa(z;\alpha)=\frac{\kappa}{\kappa_1'}\mathcal{W}_{\kappa_1'}(z;\alpha)+\frac{\kappa_1'-\kappa}{\kappa_1'}\mathcal{W}_0(z;\alpha).
\end{equation}
Combining \eqref{H1} and \eqref{H2}, one has, for $\kappa\in[0,\kappa_1']$,
\begin{equation}\label{H3}
Minimizer_{z\in\Omega_{eb}}\mathcal{W}_\kappa(z;\alpha)=i.
\end{equation}
The minimum principle in Theorem \ref{Th1} gives that
\begin{equation}\label{H4}
\min_{z\in\mathbb{H}}\mathcal{W}_\kappa(z;\alpha)=\min_{z\in\Omega_{e}}\mathcal{W}_\kappa(z;\alpha)
\end{equation}
Therefore, for $\kappa\in[0, \kappa_1']$, the location identity \eqref{H3} and minimum principle \eqref{H4}(recall that $\Omega_{e}=\Omega_{ea}\cup\Omega_{eb}$) conclude that
\begin{equation}\label{H5}
\min_{z\in\mathbb{H}}\mathcal{W}_\kappa(z;\alpha)=\min_{z\in\Omega_{ea}}\mathcal{W}_\kappa(z;\alpha).
\end{equation}
The minimizer is unique up to the group $\mathcal{G}_2$.

For $\kappa\in[\kappa_1', \kappa_2']$, by Theorem \ref{Th41},
\begin{equation}\label{H6}
Minimizer_{z\in\mathbb{H}}\mathcal{W}_\kappa(z;\alpha)=i,
\end{equation}
which is unique up the group $\mathcal{G}_2$.
For $\kappa\in[\kappa_2',1]$, by the duality Lemma \ref{l-dual} and \eqref{H5}, one has
\begin{equation}\nonumber
\min_{z\in\mathbb{H}}\mathcal{W}_\kappa(z;\alpha)=\min_{z\in\Omega_{eb}}\mathcal{W}_\kappa(z;\alpha).
\end{equation}
Therefore, the proof of Theorem \ref{Th3} is complete.

\end{proof}

%%%%%%%%%%%%%%%%%%%%%%%%%%%%%%%%%%%%%%%%%%%%%%%%%%%%%%%%%%%%%%%%%%%%%%
Finally, by Lemma \ref{Lemma43}, one has $\sigma_M\geq2$, hence $[\frac{1}{1+\sigma_M},\frac{\sigma_M}{1+\sigma_M}]\supseteq[\frac{1}{3},\frac{2}{3}]$. Theorem \ref{Th2} then  follows easily from Theorem \ref{Th41}.

\vskip0.3in
It remains to prove Lemma \ref{Lemma43}. We postpone its proof to the end of this section. Before this, we introduce some auxiliary functions and lemmas to the proof.
 From Lemma \ref{lemmaD} and (\ref{Iab}), we note that each term in $I_a(y), J_a(y)$ is non-negative. To control $I_b(y), J_b(y)$, we introduce the auxiliary functions as follows
\begin{equation}\aligned
\mu_1(y,k):&=\sum_{n=k}^\infty  e^{-\pi(n^2-1)y}\\
\mu_2(y,k):&=\sum_{n=k}^\infty n^2 e^{-\pi(n^2-1)y}\\
\nu_1(y,k):&=\sum_{n=k}^\infty e^{-2\pi n(n+1)y}\\
\nu_2(y,k):&=\sum_{n=k}^\infty (2n+1)^2 e^{-2\pi n(n+1)y}
 \endaligned\end{equation}
 for $k\geq2$, these terms are small since $k\geq2$, and they are used to express the error terms.

 Within these auxiliary functions, from Lemma \ref{lemmaD}, one can have the lower bound estimates

 \begin{lemma}\label{Lemmatype1}  We obtain the first lower bound estimate
\begin{equation}\aligned\label{Lforma}
\mathcal{L}\Big(\sum_{n,m}e^{-\pi \alpha(n^2y+\frac{m^2}{y})}\Big)
\geq\sum_{n=-1}^1\sum_{m=-1}^1\pi^2\alpha^2(\frac{m^2}{y^2}-n^2)^2e^{-\pi \alpha(n^2y+\frac{m^2}{y})}\\
-\frac{4\pi \alpha}{y}e^{-\pi \alpha y}
\Big(1+2e^{-\pi \alpha\frac{1}{y}}+Error_1(\alpha,y)\Big),
 \endaligned\end{equation}
 where

 \begin{equation}\label{eee1}
 Error_1(\alpha,y):=2\mu_2(\alpha y,2)e^{-\pi\alpha\frac{1}{y}}+2\mu_2( \alpha y,2)\mu_1( \alpha \frac{1}{y},2),
 \end{equation}
and  the second lower bound estimate
  \begin{equation}\aligned\label{Lformb}
&\mathcal{L}\Big(\sum_{n,m}e^{-\frac{\pi \alpha}{2}\big((2n+1)^2y+\frac{(2m+1)^2}{y}\big)}\Big)\\
&\geq\sum_{n=-1}^1\sum_{m=-1}^1\frac{\pi^2\alpha^2}{4}\big(\frac{(2m+1)^2}{y^2}-(2n+1)^2\big)^2e^{-\frac{\pi \alpha}{2}\big((2n+1)^2y+\frac{(2m+1)^2}{y}\big)}\\
&-\frac{4\pi \alpha}{y}e^{-\frac{\pi \alpha}{2}(y+\frac{1}{y})}(1+Error_2(\alpha,y)),
 \endaligned\end{equation}
 where
 \begin{equation}\label{eee2}
 Error_2(\alpha,y):=\nu_1(\alpha,1)+\nu_2(\alpha,1)+\nu_1(\alpha,1)\nu_2(\alpha,1).
 \end{equation}

 \end{lemma}

Before going to the precise estimates, we use the main order analysis to show the sketch of the proof.
Notice that here
 \begin{equation}\nonumber
y\in[1,\sqrt3].
 \end{equation}
Here $y$ is variable however it is well controlled by the lower and upper bounds.

Notice that in the lower bound estimates in \eqref{Lforma} and \eqref{Lformb}, the main order terms only have three
possibilities as follows
 \begin{equation}\aligned\nonumber
e^{-\pi\alpha y}, e^{-\pi\alpha \frac{1}{y}},\; e^{-\pi\alpha \frac{1}{2}(y+\frac{1}{y})}.
 \endaligned\end{equation}

To illustrate the main idea, we only capture the main order terms in  Lemma \ref{Lemmatype1}.

By Lemma \ref{Lemmatype1}, the leading order terms of $\mathcal{L}\Big(\sum_{n,m}e^{-\pi \alpha(n^2y+\frac{m^2}{y})}\Big)$ are given by
 \begin{equation}\aligned\label{VVV1}
&\mathcal{L}_m\Big(\sum_{n,m}e^{-\pi \alpha(n^2y+\frac{m^2}{y})}\Big)\\
=
&
2\pi^2\alpha^2\frac{1}{y^4}e^{-\pi\alpha \frac{1}{y}}+\big(2\pi^2\alpha^2-\frac{4\pi\alpha}{y}\big)e^{-\pi\alpha y}
 \endaligned\end{equation}
and the leading order terms of $\mathcal{L}\Big(\sum_{n,m}e^{-\pi \alpha(n^2y+\frac{m^2}{y})}\Big)$ are given by

\begin{equation}\aligned\label{VVV2}
&\mathcal{L}_m\Big(\sum_{n,m}e^{-\frac{\pi \alpha}{2}\big((2n+1)^2y+\frac{(2m+1)^2}{y}\big)}\Big)\\
&=
\pi^2\alpha^2(\frac{1}{y^2}-1)^2 e^{-\pi\alpha\frac{1}{2}(y+\frac{1}{y})}-\frac{4\pi\alpha}{y}e^{-\pi\alpha\frac{1}{2}(y+\frac{1}{y})}.
 \endaligned\end{equation}
The term $\pi^2\alpha^2(\frac{1}{y^2}-1)^2 e^{-\pi\alpha\frac{1}{2}(y+\frac{1}{y})}$ will be dropped in estimates since it is nonnegative and
it is zero as $y=1$.
Now by \eqref{VVV1} and \eqref{VVV2},
\begin{equation}\aligned\label{wwwa}
&\mathcal{L}_m\Big(\sum_{n,m}e^{-\frac{\pi \alpha}{2}\big((2n+1)^2y+\frac{(2m+1)^2}{y}\big)}\Big)+\sigma
\mathcal{L}_m\Big(\sum_{n,m}e^{-\frac{\pi \alpha}{2}\big((2n+1)^2y+\frac{(2m+1)^2}{y}\big)}\Big)\\
\geq&
2\pi^2\alpha^2\frac{1}{y^4}e^{-\pi\alpha \frac{1}{y}}+\big(2\pi^2\alpha^2-\frac{4\pi\alpha}{y}\big)e^{-\pi\alpha y}
-\sigma\frac{4\pi\alpha}{y}e^{-\pi\alpha\frac{1}{2}(y+\frac{1}{y})}.
 \endaligned\end{equation}
 To simplify the notations, let
\begin{equation}\aligned\label{Palpha}
\mathcal{P}(\alpha,y;\sigma):=
2\pi^2\alpha^2\frac{1}{y^4}e^{-\pi\alpha \frac{1}{y}}+\big(2\pi^2\alpha^2-\frac{4\pi\alpha}{y}\big)e^{-\pi\alpha y}
-\sigma\frac{4\pi\alpha}{y}e^{-\pi\alpha\frac{1}{2}(y+\frac{1}{y})}.
 \endaligned\end{equation}
 Thus
 \begin{equation}\aligned\label{wwwaa}
\mathcal{L}_m\Big(\sum_{n,m}e^{-\frac{\pi \alpha}{2}\big((2n+1)^2y+\frac{(2m+1)^2}{y}\big)}\Big)+\sigma
\mathcal{L}_m\Big(\sum_{n,m}e^{-\frac{\pi \alpha}{2}\big((2n+1)^2y+\frac{(2m+1)^2}{y}\big)}\Big)
\geq\mathcal{P}(\alpha,y;\sigma).
 \endaligned\end{equation}
 We shall prove that $\mathcal{P}(\alpha,y;\sigma)>0$ for some positive $\sigma$.

 Since $y\geq1,\alpha\geq1$, one has $\frac{4\pi\alpha}{y}\leq\frac{2}{\pi}2\pi^2\alpha^2$, thus

 \begin{equation}\aligned\nonumber
2\pi^2\alpha^2-\frac{4\pi\alpha}{y}\geq(1-\frac{2}{\pi})2\pi^2\alpha^2.
 \endaligned\end{equation}
 Then
 \begin{equation}\aligned\label{wwwb}
\mathcal{P}(\alpha,y;\sigma)
\geq
2\pi^2\alpha^2\frac{1}{y^4}e^{-\pi\alpha \frac{1}{y}}+(1-\frac{2}{\pi})2\pi^2\alpha^2e^{-\pi\alpha y}
-\sigma\frac{4\pi\alpha}{y}e^{-\pi\alpha\frac{1}{2}(y+\frac{1}{y})}.
 \endaligned\end{equation}
 By the basic mean value inequality, one has, for any positive $\mathcal{X}, \mathcal{Y}$, there has
  \begin{equation}\aligned\nonumber
\mathcal{X}e^{-\pi\alpha y} +\mathcal{Y} e^{-\pi\alpha \frac{1}{y}}\geq 2\sqrt{\mathcal{X}\mathcal{Y}}e^{-\pi\alpha \frac{1}{2}(y+\frac{1}{y})},
 \endaligned\end{equation}
 with equality holds only if $y=1$ and $\mathcal{X}=\mathcal{Y}$.
 It follows that, at least for
   \begin{equation}\aligned\label{sigma000}
\sigma=\sqrt{1-\frac{2}{\pi}}\pi,
 \endaligned\end{equation}
 there holds that
  \begin{equation}\aligned\nonumber
2\pi^2\alpha^2\frac{1}{y^4}e^{-\pi\alpha \frac{1}{y}}+(1-\frac{2}{\pi})2\pi^2\alpha^2e^{-\pi\alpha y}
>\sigma\frac{4\pi\alpha}{y}e^{-\pi\alpha\frac{1}{2}(y+\frac{1}{y})},
 \endaligned\end{equation}
 which implies that
   \begin{equation}\aligned\nonumber
\mathcal{P}(\alpha,y;\sigma)>0.
 \endaligned\end{equation}

 In view of \eqref{wwwb}, for the value of $\sigma$ in \eqref{sigma000}, one has
\begin{equation}\aligned\label{wwwbb}
\mathcal{L}_m\Big(\sum_{n,m}e^{-\frac{\pi \alpha}{2}\big((2n+1)^2y+\frac{(2m+1)^2}{y}\big)}\Big)+\sigma
\mathcal{L}_m\Big(\sum_{n,m}e^{-\frac{\pi \alpha}{2}\big((2n+1)^2y+\frac{(2m+1)^2}{y}\big)}\Big)
\geq0.
 \endaligned\end{equation}
By the main order analysis, we give the heuristic proof of Lemma \ref{Lemma43}.
In the next, we give full proof of Lemma \ref{Lemma43} and provide the scheme of finding the optimal value of $\sigma$.

We are now in a position to prove Lemma \ref{Lemma43}.
\vskip0.3in
{\bf Proof of Lemma \ref{Lemma43}.}

\begin{proof} We only need to consider the cases $\alpha\geq1$ and $y\geq1$ by Lemma \ref{l-dual} and identity \eqref{JJJ1} respectively.

By Lemma \ref{Lemmatype1}, one has

\begin{equation}\aligned\label{wwwa}
&\mathcal{L}\Big(\sum_{n,m}e^{-\frac{\pi \alpha}{2}\big((2n+1)^2y+\frac{(2m+1)^2}{y}\big)}\Big)+\sigma
\mathcal{L}\Big(\sum_{n,m}e^{-\frac{\pi \alpha}{2}\big((2n+1)^2y+\frac{(2m+1)^2}{y}\big)}\Big)\\
\geq&
\sum_{n=-1}^1\sum_{m=-1}^1\pi^2\alpha^2(\frac{m^2}{y^2}-n^2)^2e^{-\pi \alpha(n^2y+\frac{m^2}{y})}
-\frac{4\pi \alpha}{y}e^{-\pi \alpha y}
\Big(1+2e^{-\pi \alpha\frac{1}{y}}\Big)\\
+&\sigma
\Big(
\sum_{n=-1}^1\sum_{m=-1}^1\frac{\pi^2\alpha^2}{4}\big(\frac{(2m+1)^2}{y^2}-(2n+1)^2\big)^2e^{-\frac{\pi \alpha}{2}\big((2n+1)^2y+\frac{(2m+1)^2}{y}\big)}\\
&-\frac{4\pi \alpha}{y}e^{-\frac{\pi \alpha}{2}(y+\frac{1}{y})})
\Big)\\
-& \mathcal{R}_1(\alpha,y).
\endaligned\end{equation}
Here
\begin{equation}\aligned\nonumber
\mathcal{R}_1(\alpha,y)=\frac{4\pi \alpha}{y}e^{-\pi \alpha y}
Error_1(\alpha,y)+
\frac{4\pi \alpha}{y}e^{-\frac{\pi \alpha}{2}(y+\frac{1}{y})}Error_2(\alpha,y),
 \endaligned\end{equation}
where
$Error_1(\alpha,y)$ and $Error_2(\alpha,y)$ are defined in \eqref{eee1} and \eqref{eee2} respectively.
Now we define
\begin{equation}\aligned\label{wwwa}
\mathcal{A}(\alpha,y;\sigma):=&
\sum_{n=-1}^1\sum_{m=-1}^1\pi^2\alpha^2(\frac{m^2}{y^2}-n^2)^2e^{-\pi \alpha(n^2y+\frac{m^2}{y})}
-\frac{4\pi \alpha}{y}e^{-\pi \alpha y}
\Big(1+2e^{-\pi \alpha\frac{1}{y}}\Big)\\
+&\sigma
\Big(
\sum_{n=-1}^1\sum_{m=-1}^1\frac{\pi^2\alpha^2}{4}\big(\frac{(2m+1)^2}{y^2}-(2n+1)^2\big)^2e^{-\frac{\pi \alpha}{2}\big((2n+1)^2y+\frac{(2m+1)^2}{y}\big)}\\
&-\frac{4\pi \alpha}{y}e^{-\frac{\pi \alpha}{2}(y+\frac{1}{y})})
\Big).
\endaligned\end{equation}
Then
\begin{equation}\aligned\label{wwwa}
&\mathcal{L}\Big(\sum_{n,m}e^{-\frac{\pi \alpha}{2}\big((2n+1)^2y+\frac{(2m+1)^2}{y}\big)}\Big)+\sigma
\mathcal{L}\Big(\sum_{n,m}e^{-\frac{\pi \alpha}{2}\big((2n+1)^2y+\frac{(2m+1)^2}{y}\big)}\Big)\\
\geq&
\mathcal{A}(\alpha,y;\sigma)-\mathcal{R}_1(\alpha,y).
\endaligned\end{equation}

The remainder term $\mathcal{R}_1(\alpha,y)$ is very small comparing the major term $\mathcal{A}(\alpha,y;\sigma)$ for $\sigma\geq0$.
We note that the estimate in $\mathcal{A}(\alpha,y;\sigma)$ is much refined than that in the heuristic estimate. In fact, after dropping the nonnegative terms in $\mathcal{A}(\alpha,y;\sigma)$, one has
\begin{equation}\aligned\label{hhhddd}
\mathcal{A}(\alpha,y;\sigma)
\geq&
2\pi^2\alpha^2\frac{1}{y^4}e^{-\pi\alpha \frac{1}{y}}+\big(2\pi^2\alpha^2-\frac{4\pi\alpha}{y}(1+2e^{-\pi \alpha\frac{1}{y}})\big)e^{-\pi\alpha y}
-\sigma\frac{4\pi\alpha}{y}e^{-\pi\alpha\frac{1}{2}(y+\frac{1}{y})}\\
=&\mathcal{P}(\alpha,y;\sigma)-\frac{8\pi\alpha}{y}e^{-\pi\alpha(y+\frac{1}{y})}.
 \endaligned\end{equation}
 The lower bound is quite similar to the one in $\mathcal{P}(\alpha,y;\sigma)$(\eqref{wwwb}).
Same as the estimates in \eqref{wwwb}, similar to calculation of the value of $\sigma$ in \eqref{sigma000}, i.e., for
 \begin{equation}\aligned\nonumber
\sigma=\sqrt{1-\frac{2}{\pi}}\pi-2e^{-\pi},
 \endaligned\end{equation}
one has
\begin{equation}\aligned\label{hhheee}
2\pi^2\alpha^2\frac{1}{y^4}e^{-\pi\alpha \frac{1}{y}}+\big(2\pi^2\alpha^2-\frac{4\pi\alpha}{y}(1+2e^{-\pi \alpha\frac{1}{y}})\big)e^{-\pi\alpha y}
-\sigma\frac{4\pi\alpha}{y}e^{-\pi\alpha\frac{1}{2}(y+\frac{1}{y})}>0
 \endaligned\end{equation}
for $\alpha\geq1$ and $y\geq1$. Hence by \eqref{hhhddd} and \eqref{hhheee},
\begin{equation}\aligned\nonumber
\mathcal{A}(\alpha,y;\sqrt{1-\frac{2}{\pi}}\pi-2e^{-\pi})>0\;\;\hbox{for}\;\; \alpha\geq1, y\in[1,\sqrt3].
\endaligned\end{equation}

The value of $\sigma$ such that $\mathcal{A}(\alpha,y;\sigma)$ is positive on the strip $\alpha\geq1$ and $y\in[1,\sqrt3]$ can be improved by the auxiliary function $\mathcal{A}(\alpha,y;\sigma)$.

In the next, we show that
\begin{equation}\aligned\label{Akey}
\mathcal{A}(\alpha,y;2)>0\;\;\hbox{for}\;\; \alpha\geq1, y\in[1,\sqrt3].
\endaligned\end{equation}
To use the monotonicity, we set
\begin{equation}\aligned\label{AB}
\mathcal{A}(\alpha,y;\sigma)
=e^{-\pi\alpha y} y^{-4}
\mathcal{B}(\alpha,y;\sigma),
\endaligned\end{equation}
where
\begin{equation}\aligned\nonumber
\mathcal{B}(\alpha,y;\sigma)
=&2\pi^2\alpha^2 y^4+2\pi^2\alpha^2 e^{\pi\alpha(y-\frac{1}{y})}
+4\pi^2\alpha^2(y^2-1)^2e^{-\pi\alpha\frac{1}{y}}\\
&-4\pi\alpha y^3-8\pi\alpha y^3 e^{-\pi\alpha\frac{1}{y}}-
\sigma\Big(
\pi^2\alpha^2(y^2-1)^2e^{\frac{\pi\alpha}{2}(y-\frac{1}{y})}-4\pi\alpha y^3 e^{\frac{\pi\alpha}{2}(y-\frac{1}{y})}
\Big).
\endaligned\end{equation}

By an elementary but complicated argument, one has
\begin{equation}\aligned\nonumber
\frac{\partial}{\partial \alpha}\mathcal{B}(\alpha,y;\sigma)>0, \frac{\partial}{\partial y}\mathcal{B}(\alpha,y;\sigma)>0\;\;
\hbox{for}\;\; y\in[1,\sqrt3], \alpha\geq1, \sigma\in[0,2].
\endaligned\end{equation}
It follows that
\begin{equation}\aligned\nonumber
\min_{\alpha\geq1,y\in[1,\sqrt3]}\mathcal{B}(\alpha,y;2)=\mathcal{B}(1,1;2).
\endaligned\end{equation}
Now the numerical result shows that
\begin{equation}\aligned\nonumber
\mathcal{B}(1,1;2)>0.
\endaligned\end{equation}
Therefore by \eqref{AB}, \eqref{Akey} is proved.

Since the remainder term $\mathcal{R}_1(\alpha,y)$ is so small comparing the major term $\mathcal{A}(\alpha,y;\sigma)$,  one has
\begin{equation}\aligned\nonumber
\mathcal{A}(\alpha,y;2)-\mathcal{R}_1(\alpha,y)>0\;\;\hbox{for}\;\; \alpha\geq1, y\in[1,\sqrt3].
\endaligned\end{equation}
By \eqref{wwwa}, the proof is complete.

\end{proof}

\medskip

\noindent
{\bf Remark on the Proof of Lemma \ref{Lemma43}.}

The bound $\sigma=2$ is not sharp as we can see in the proof. However we can improve this bound by some  auxiliary functions.
By taking more  terms in the identity of Lemma \ref{lemmaD}, one can get more accurate upper bound.

\vskip0.2in

\noindent
{\bf Acknowledgment:}
The research of S. Luo was partially supported by double thousand plan of Jiangxi(jxsq2019101048) and NSFC(No. 12001253).
 The research of J. Wei was partially supported by NSERC of Canada.
The research of W. Zou was partially supported by NSFC(Nos. 11801581,11025106, 11371212, 11271386).
S. Luo is grateful to Prof.H.J. Zhao (Wuhan University)  for his constant support and encouragement.

\section{ Appendix: The proof of Lemma \ref{LemmaT3}}

\setcounter{equation}{0}

 We separate the proof into two parts.
On the one hand, by the Jacobi tri-product formula, we have
\begin{equation}\aligned\label{Product}
\vartheta(X;Y)&=\prod_{n=1}^\infty(1-e^{-2\pi n X})(1+e^{-(2n-1)\pi X}e^{2\pi iY})(1+e^{-(2n-1)\pi X}e^{-2\pi iY})\\
&=\prod_{n=1}^\infty(1-e^{-2\pi n X})(1+e^{-2(2n-1)\pi X}+2e^{-(2n-1)\pi X}\cos(2\pi Y)).
 \endaligned\end{equation}

Taking logarithmic on both sides of \eqref{Product} and differentiating with respect to $Y$, one gets
\begin{equation}\aligned\label{QQQ}
-\frac{\frac{\partial}{\partial Y}\vartheta(X;Y)}{\sin(2\pi Y)}&=4\pi\sum_{n=1}^\infty e^{-(2n-1)\pi X}\frac{\vartheta(X;Y)}{1+e^{-2(2n-1)\pi X}+2e^{-(2n-1)\pi X}\cos(2\pi Y)}\\
&=4\pi\sum_{n=1}^\infty e^{-(2n-1)\pi X}\prod_{m\neq n,m=1}^\infty(1-e^{-2\pi m X})(1+e^{-2(2m-1)\pi X}+2e^{-(2m-1)\pi X}\cos(2\pi Y)).
 \endaligned\end{equation}
One sees from \eqref{QQQ} that the function $-\frac{\frac{\partial}{\partial Y}\vartheta(X;Y)}{\sin(2\pi Y)}$ has a period $1$, is decreasing on $[0,\frac{1}{2}]$
and a even function for $Y$, just like $\cos(2\pi Y)$.

Then
\begin{equation}\aligned\label{QQQ1}
\lim_{Y\rightarrow\frac{1}{2}}-\frac{\frac{\partial}{\partial Y}\vartheta(X;Y)}{\sin(2\pi Y)}&\leq-\frac{\frac{\partial}{\partial Y}\vartheta(X;Y)}{\sin(2\pi Y)}\leq\lim_{Y\rightarrow0}-\frac{\frac{\partial}{\partial Y}\vartheta(X;Y)}{\sin(2\pi Y)}.
\endaligned\end{equation}
 It follows that
\begin{equation}\aligned\label{QQQ2}
\frac{1}{2\pi}\frac{\partial^2}{\partial Y^2}\vartheta(X;Y)\mid_{Y=\frac{1}{2}}\leq-\frac{\frac{\partial}{\partial Y}\vartheta(X;Y)}{\sin(2\pi Y)}\leq-\frac{1}{2\pi}\frac{\partial^2}{\partial Y^2}\vartheta(X;Y)\mid_{Y=0}\\
\endaligned\end{equation}
by L'Hospital rule.

By \eqref{TXY}, one has,
\begin{equation}\aligned\label{QQQ3}
\frac{\partial^2}{\partial Y^2}\vartheta(X;Y)\mid_{Y=0}&=4\pi e^{-\pi X}(1+\sum_{n=2}^\infty n^2 e^{-\pi(n^2-1) X})\\
\frac{1}{2\pi}\frac{\partial^2}{\partial Y^2}\vartheta(X;Y)\mid_{Y=\frac{1}{2}}&=4\pi\sum_{n=1}^\infty (-1)^{n-1}n^2 e^{-n^2\pi X}
\geq4\pi e^{-\pi X}(1-\sum_{n=2}^\infty n^2 e^{-\pi(n^2-1) X}).
\endaligned\end{equation}
Combining \eqref{QQQ1}, \eqref{QQQ2} and \eqref{QQQ3}, one gets the result for lower and upper for $ t>a>0$. For more details see  \cite{LW}.

On the other hand, for t is small and closes to 0, by Poisson Summation Formula, one has
\begin{equation}\aligned\nonumber
\vartheta(X;Y)=X^{-\frac{1}{2}}\sum_{n\in \mathbb{Z}} e^{-\frac{\pi(n-Y)^2}{X}}.
\endaligned\end{equation}

We have
\begin{equation}\aligned\nonumber
\frac{\partial}{\partial Y}\vartheta(t;Y)&=-2\pi t^{-\frac{3}{2}}\sum_{m=-\infty}^\infty e^{-\frac{\pi(m-Y)^2}{t}}(Y-m)\\
&=-2\pi t^{-\frac{3}{2}}(\sum_{m=0}^\infty e^{-\frac{\pi (k+m)^2}{t}}(k+m-e^{-\frac{\pi(2m+1)(1-2k)}{t}}(m+1-k)))
\endaligned
\end{equation}
Notice that for $Y\leq 1, t\leq 2\pi$
\begin{equation}\aligned\nonumber
\frac{d}{dY}\big(
e^{-\frac{\pi(2m+1)(1-2Y)}{t}(m+1-Y)}\big)=\frac{e^{-\frac{\pi(2m+1)(1-2Y)}{t}}}{t}(4m^2\pi+6m\pi+2\pi-4Y m\pi-2Y\pi-t)>0.
\endaligned
\end{equation}
Then by this, one gets the  lower and upper bounds for $ 0<t<k\leq \frac{\pi}{2}$. The proof of Lemma \ref{LemmaT3} is thus complete.

\section{Appendix: the proof of Theorem \ref{Th4} }

\setcounter{equation}{0}

We start with the scaling (by parameter a) of the Gamma function
\begin{equation}\label{Gamma}
\frac{\Gamma(s)}{\pi^s}\frac{1}{a^s}=\int_0^\infty e^{-\pi a t}t^s\frac{dt}{t}
\end{equation}
which can be rewritten in convergent form
\begin{equation}\aligned\nonumber
\frac{\Gamma(s)}{\pi^s}\frac{1}{a^s}=(\frac{1}{s-1}-\frac{1}{s})
+\int_1^\infty(e^{-\pi a t}-1)(t^s+t^{1-s})\frac{dt}{t}.
\endaligned
\end{equation}
It follows that
\begin{equation}\aligned\label{GLW}
\mathbf{G}_R (z;\kappa,s)=(\frac{1}{s-1}-\frac{1}{s})
+\int_1^\infty(\mathcal{W}_\kappa(z;\alpha)-1)(t^s+t^{1-s})\frac{dt}{t}.
\endaligned
\end{equation}
See $\mathbf{G}_R (z;\kappa,s)$ and $\mathcal{W}_\kappa(z;\alpha)$ in \eqref{EisenW} and \eqref{Wzalphadef} respectively.
The identity \eqref{GLW} can be viewed as an analogue of the Bernstein Theorem in Theorem \eqref{Bern}.
Taking derivative with respect to $s$, one has
\begin{equation}\aligned\label{GLWD}
\frac{d^j}{ds^j}\mathbf{G}_R (z;\kappa,s)=\frac{d^j}{ds^j}(\frac{1}{s-1}-\frac{1}{s})
+\int_1^\infty(\mathcal{W}_\kappa(z;\alpha)-1)\frac{d^j}{ds^j}(t^s+t^{1-s})\frac{dt}{t}.
\endaligned
\end{equation}
Notice that
\begin{equation}\aligned\nonumber
\frac{d^j}{ds^j}(t^s+t^{1-s})=(\log t)^j t^s+(-\log t)^j t^{1-s}
\endaligned
\end{equation}
is nonnegative for any $j\geq1, t\geq1, s\geq\frac{1}{2}$.

By the bridge identity \eqref{GLWD}, Theorem \ref{Th4} can be proved by following similar proofs  of Theorem \ref{Th1} and Theorem \ref{Th2}. We omit the details.


\begin{thebibliography}{10}



\bibitem{Apo1976}
T. M. Apostol. Modular functions and Dirichlet series in number theory. Springer-Verlag, Berlin
Heidelberg, 1976.



\bibitem{Alzer2002}
 H. Alzer; C. Berg, Some classes of completely monotonic functions. {\em Ann. Acad. Sci. Fenn. Math.} 27 (2002), no. 2, 445-460.

\bibitem{Alzer2006}
 H. Alzer; C. Berg, Some classes of completely monotonic functions. II. {\em Ramanujan J.} 11 (2006), no. 2, 225-248.


\bibitem{Askey1974}
R. Askey; H. Pollard, Some absolutely monotonic and completely monotonic functions. {\em SIAM J. Math. Anal.} 5 (1974), 58-63.


\bibitem{Bernstein}
S. Bernstein, Sur les fonctions absolument monotones. (French) {\em Acta Math.} 52 (1929), no. 1, 1-66.

\bibitem{Betermin2016}
L. B\'{e}termin, Two-dimensional theta functions and crystallization among Bravais lattices. {\em SIAM J. Math. Anal.} 48 (2016), no. 5, 3236-3269.

 \bibitem{Betermin2021} L. B\'{e}termin, M. Faulhuber and H. Knüpfer,  On the optimality of the rock-salt structure among lattices with charge distributions. {\em Math. Models Methods Appl. Sci. } 31 (2021), no. 2, 293–325.


\bibitem{Betermin2021-2}
Laurent B\'{e}termin, Markus Faulhuber, Stefan Steinerberger, A variational principle for Gaussian lattice sums, 	arXiv:2110.06008.

\bibitem{Cohen}
 H. Cohen, Number theory. Vol. II. Analytic and modern tools. Graduate Texts in Mathematics, 240. Springer, New York, 2007. xxiv$+$596 pp. ISBN: 978-0-387-49893-5.


\bibitem{Cohn2006JAMS}
H. Cohn and A. Kumar, Universally optimal distribution of points on spheres,
{\em Journal of the American Mathematical Society} 20 (2007), 99-148.


\bibitem{Eva1973}
R. Evans. A fundamental region for Hecke modular group. {\em J. Number Theory}, 5(2) (1973), 108-115.

\bibitem{Ismail1986}
M. Ismail; L. Lorch; M. Muldoon, Completely monotonic functions associated with the gamma function and its q-analogues. {\em J. Math. Anal. Appl.} 116 (1986), no. 1, 1-9.

\bibitem{Miller2001}
K. S. Mille; S. G.Samko, Completely monotonic functions. (English summary),
{\em Integral Transform. Spec. Funct.} 12 (2001), no. 4, 389-402.


\bibitem{LRW}
 S. Luo, X. Ren and J. Wei, Non-hexagonal lattices from a two species interacting system, {\em SIAM J. Math. Anal.}, 52(2) (2020), 1903-1942.

\bibitem{LW} S. Luo and J. Wei,
On minima of  sum of theta functions and application to Mueller-Ho Conjecture, 	{\em Arch. Rat. Mech. Anal.}, to appear.


\bibitem{Mon1988}
H. Montgomery, Minimal theta functions. {\em Glasgow Math. J.} 30 (1988), 75-85.


\bibitem{Mue2002}E.J. Mueller and T.L. Ho,
Two-component Bose-Einstein condensates with a large number of vortices,
{\em Physical review letters}, 88 (2002), 180403.



\bibitem{Serfaty2012}
E. Sandier and S. Serfaty, From the Ginzburg-Landau model to vortex lattice problems. {\em Comm. Math. Phys.} 313(2012), 635-743.

\bibitem{Siegel}
C.L. Siegel, Lectures on advanced analytic number theory. Notes by S. Raghavan. Tata Institute of Fundamental Research Lectures on Mathematics, No. 23 Tata Institute of Fundamental Research, Bombay 1965 iii+331+iii pp.


\bibitem{Song2012}
Schilling, Ren$\acute{e}$L.; R. Song; Vondra$\breve{c}$ek, Zoran Bernstein functions. Theory and applications. Second edition. De Gruyter Studies in Mathematics, 37. Walter de Gruyter Co., Berlin, 2012. xiv$+$410 pp. ISBN: 978-3-11-025229-3; 978-3-11-026933-8.


\end{thebibliography}
\end{document}